\documentclass[11pt]{amsart}

\usepackage{amsthm,amsfonts,amsmath,amssymb,verbatim,rotating} 
\usepackage[left=2.9cm,right=2.9cm,top=3.5cm,bottom=3.5cm]{geometry}
\usepackage{tikz,graphicx}
\usepackage{mathrsfs}
\usepackage{braket}

\usepackage{fouridx}

\usepackage[normalem]{ulem}

\usepackage{hyperref}

\allowdisplaybreaks

\numberwithin{equation}{section}
\newtheorem{theorem}{Theorem}[section]
\newtheorem{corollary}[theorem]{Corollary}

\newtheorem{proposition}[theorem]{Proposition}

\theoremstyle{definition}

\renewcommand{\leq}{\leqslant}
\renewcommand{\geq}{\geqslant}

\renewcommand{\Re}{\textup{Re}}

\newcommand{\dup}{\mathrm{d}}

\newcommand{\iup}{\hspace{1pt}\mathrm{i}\hspace{1pt}}
\newcommand{\Int}{\int\limits}
\newcommand{\bla}{\boldsymbol{\la}}
\newcommand{\bmu}{\boldsymbol{\mu}}
\newcommand{\bnu}{\boldsymbol{\nu}}

\newcommand{\bkappa}{\boldsymbol{\kappa}}
\newcommand{\bzero}{\boldsymbol{0}}
\newcommand{\xar}[1]{x^{(#1)}}
\newcommand{\zar}[1]{z^{(#1)}}

\newcommand{\Symm}{\mathfrak{S}}
\newcommand{\abs}[1]{\lvert#1\rvert}

\newcommand{\la}{\lambda}

\newcommand{\obinomE}{\genfrac\langle\rangle{0pt}{}}

\newcommand{\spec}[1]{\langle #1\rangle}
\newcommand{\Spec}[1]{\big\langle #1\big\rangle}
\newcommand{\A}{\mathrm A}
\newcommand{\mur}[1]{\mu^{(#1)}}
\newcommand{\lar}[1]{\lambda^{(#1)}}
\newcommand{\DeltaSv}{\Delta_{\mathrm{S}}^{(\mathrm{v})}}
\newcommand{\DeltaSe}{\Delta_{\mathrm{S}}^{(\mathrm{e})}}
\newcommand{\DeltaS}{\Delta_{\mathrm{S}}}
\newcommand{\DeltaD}{\Delta_{\mathrm{D}}}
\newcommand{\Part}{\mathscr{P}}
\newcommand{\Gammapq}{\Gamma_{p,q}}

\newcommand{\thetaq}{\theta_q}
\newcommand{\thetap}{\theta_p}
\newcommand{\Ri}{R^{\ast}}

\begin{document}

\title[Elliptic Selberg integrals]{Elliptic $\A_n$ Selberg integrals}
\author{Seamus P. Albion}
\address{Fakult\"{a}t f\"{u}r Mathematik, Universit\"{a}t Wien, 
Oskar-Morgenstern-Platz 1, A-1090 Vienna, Austria}
\email{seamus.albion@univie.ac.at}

\author{Eric M. Rains}
\address{Department of Mathematics, California Institute of Technology,
Pasadena, CA 91125, USA} 
\email{rains@caltech.edu}
 
\author{S. Ole Warnaar}                                                         
\address{School of Mathematics and Physics,
The University of Queensland, Brisbane, QLD 4072,
Australia}
\email{o.warnaar@maths.uq.edu.au}

\begin{abstract}
We use the elliptic interpolation kernel due to the second 
author to prove an $\A_n$ extension of the elliptic Selberg integral.
More generally, we obtain elliptic analogues of the $\A_n$ Kadell,
Hua--Kadell and Alba--Fateev--Litvinov--Tarnopolsky (or AFLT) integrals.
\end{abstract}

\maketitle

\section{Introduction}

In his famous 1944 paper \cite{Selberg44}, Atle Selberg evaluated the following 
multivariate extension of Euler's beta integral that now bears his name.
For $k$ a positive integer,
\begin{align}\label{Eq_Selberg}
S_k(\alpha,\beta;\gamma)&:=\Int_{[0,1]^k}\!
\prod_{i=1}^k x_i^{\alpha-1}(1-x_i)^{\beta-1}
\prod_{1\leq i<j\leq k}\abs{x_i-x_j}^{2\gamma}\,\dup x_1\cdots\dup x_k \\
&\hphantom{:}=\prod_{i=1}^k
\frac{\Gamma(\alpha+(i-1)\gamma)\Gamma(\beta+(i-1)\gamma)
\Gamma(1+i\gamma)}{\Gamma(\alpha+\beta+(k+i-2)\gamma)\Gamma(1+\gamma)},
\notag
\end{align}
where $\alpha,\beta,\gamma\in\mathbb{C}$ such that
$\Re(\alpha)>0$, $\Re(\beta)>0$ and
\[
\Re(\gamma)>-\min\{1/k,\Re(\alpha)/(k-1),\Re(\beta)/(k-1)\}.
\]
The Selberg integral has come to be regarded as one of the most
fundamental hypergeometric integrals, a reputation which is upheld by
its appearance in numerous different areas of mathematics such as 
random matrix theory \cite{BBAP05,Forrester10,FR05,Mehta04}, 
analytic number theory \cite{Anderson90,Evans91,FZ19,KS00a,KS00b},
enumerative combinatorics \cite{KNPV15,KO17,KS17,Stanley11}, and conformal
field theory \cite{AFLT11,DF84,EFJ98,MMS12,MV00,SV91,TV03,TV20,Varchenko03}.
For a review of the history and mathematics surrounding Selberg's integral
the reader is referred to~\cite{FW08}.

There are many important generalisations of the Selberg integral.
One of the goals of this paper is to unify most of these by proving an
elliptic analogue of the Selberg integral for the Lie algebra $\A_n$,
as well as elliptic analogues of the more general Kadell, Hua--Kadell
and AFLT integrals for $\A_n$. 
Before we describe the first of these generalisations, we remind the reader
of the elliptic analogue of the ordinary (or $\A_1$) Selberg integral and
of the (non-elliptic) $\A_n$ Selberg integral.

Fix $p,q\in\mathbb{C}$ such that $\abs{p},\abs{q}<1$, and let
\[
\Gammapq(z):=\prod_{i,j=0}^{\infty}\frac{1-p^{i+1}q^{j+1}/z}{1-p^iq^jz}
\]
be the elliptic Gamma function~\cite{Ruijsenaars97}.
This function, which has zeros at $p^{\mathbb{N}_0+1}q^{\mathbb{N}_0+1}$,
poles at $p^{-\mathbb{N}_0}q^{-\mathbb{N}_0}$ and an essential
singularity at the origin, is symmetric in $p$ and $q$ and satisfies the
reflection formula
\begin{equation}\label{Eq_reflection}
\Gammapq(z)\Gammapq(pq/z)=1.
\end{equation}
As is by now standard, in the following 
we adopt the multiplicative shorthand notation $\Gammapq(z_1,\dots,z_n):=
\Gammapq(z_1)\cdots\Gammapq(z_n)$ as well as the plus-minus notation
\begin{align*}
\Gammapq(az^{\pm})&:=\Gammapq(az,az^{-1}), \\
\Gammapq(az^{\pm}w^{\pm})&:=
\Gammapq(azw,az^{-1}w,azw^{-1},az^{-1}w^{-1}).
\end{align*}
Again assuming that $\abs{q}<1$, let
$(a;q)_{\infty}:=\prod_{i\geq 0}(1-aq^i)$ be the 
infinite $q$-shifted factorial.
Then the elliptic Selberg density is defined as \cite{Rains09,Spiridonov03}
\begin{equation}\label{Eq_Selberg-density}
\DeltaSv(z_1,\dots,z_k;t_1,\dots,t_m;t;p,q)
:=\varkappa_k\prod_{1\leq i<j\leq k}\frac{\Gammapq(tz_i^{\pm}z_j^{\pm})}
{\Gammapq(z_i^{\pm}z_j^{\pm})}
\prod_{i=1}^k\frac{\Gammapq(t)\prod_{r=1}^m\Gammapq(t_rz_i^{\pm})}
{\Gammapq(z_i^{\pm 2})},
\end{equation}
where  $z_1,\dots,z_k,t,t_1,\dots,t_m\in\mathbb{C}^{\ast}$ and \begin{equation}\label{Eq_kappa-n}
\varkappa_k:=\frac{(p;p)^k_{\infty}(q;q)^k_{\infty}}{2^kk!(2\pi\iup)^k}.
\end{equation}
The use of the superscript $(\textrm{v})$ is non-standard.
Later we also need a companion density $\DeltaSe(\dots;\dots;c;p,q)$,
and the superscripts $(\textrm{v})$
and $(\textrm{e})$ --- $\textrm{v}$ for vertex and $\textrm{e}$ for edge
of the $\A_n$ Dynkin diagram --- have been added to avoid confusion.
Assuming $0<\abs{t},\abs{t_1},\dots,\abs{t_6}<1$ as well as the
balancing condition $t^{2k-2}t_1\cdots t_6=pq$, the elliptic Selberg
integral corresponds to
\begin{equation}\label{Eq_ESelberg}
\Int_{\mathbb{T}^k}\!  \DeltaSv(z_1,\dots,z_k;t_1,\dots,t_6;t;p,q)
\,\frac{\dup z_1}{z_1}\cdots\frac{\dup z_k}{z_k}
=\prod_{i=1}^k\bigg(\Gammapq(t^i)
\prod_{1\leq r<s\leq 6}\Gammapq(t^{i-1}t_rt_s)\bigg),
\end{equation}
where $\mathbb{T}^k$ denotes the complex $k$-torus.
For $k=1$ the above integral is Spiridonov's elliptic beta
integral~\cite{Spiridonov01}. 
For general $k$ the integral evaluation \eqref{Eq_ESelberg} was conjectured
by van Diejen and Spiridonov \cite{vDS00,vDS01} and proved by the 
second author~\cite{Rains10}.
Alternative proofs have since been given by Spiridonov \cite{Spiridonov07} 
and by Ito and Noumi~\cite{IN17}.
A rigorous proof that \eqref{Eq_ESelberg} simplifies to the Selberg integral
\eqref{Eq_Selberg} upon taking appropriate limits was presented in~\cite{Rains09}.

Elliptic beta and Selberg integrals are not just of interest from a special
functions point of view, corresponding to the top-level results in the
classical--basic--elliptic hierarchy of hypergeometric integrals.
In 2009 Dolan and Osborn \cite{DO09} made the important discovery that 
supersymmetric indices of supersymmetric $4$-dimensional quantum field
theories take the form of elliptic hypergeometric integrals. 
As a consequence, many conjectural Seiberg dualities for such quantum field
theories imply transformation formulae for the corresponding indices, and
hence for elliptic hypergeometric integrals.
Since this discovery, elliptic hypergeometric integrals and their transformation
properties play an important role in the study of dualities in quantum field
theory, see e.g.,
\cite{GPRR10,GRRY10,NNR21,RR17,SV11,SV12,SV14,SW06}.
Another surprising application of elliptic hypergeometric integrals
--- not unrelated to the supersymmetric dualities, see the survey
\cite{Gahramanov22} --- has been the construction of novel Yang--Baxter solvable
models with continuous spin parameters \cite{BKS13,BS12a,BS12b,Spiridonov10}, 
generalising many famous exactly solvable discrete spin models such as the Ising
and chiral Potts models.
These connections between elliptic hypergeometric integrals and quantum field
theory and integrable systems provide further motivation for generalising
the integral evaluation \eqref{Eq_ESelberg} to $\A_n$.

To succinctly describe the non-elliptic $\A_n$ Selberg integral, we define
\[
\Delta(x):=\prod_{1\leq i<j\leq k}(x_i-x_j) \quad\text{and}\quad
\Delta(x;y):=\prod_{i=1}^k \prod_{j=1}^{\ell} (x_i-y_j)
\]
for $x=(x_1,\dots,x_k)$ and $y=(y_1,\dots,y_{\ell})$.
Let $0=:k_0\leq k_1\leq \cdots\leq k_n$ be nonnegative integers and,
for $1\leq r\leq n$, denote by
$\xar{r}=\big(\xar{r}_1,\dots,\xar{r}_{k_r}\big)$ a $k_r$-tuple of
integration variables.
Further let $\alpha_1,\dots,\alpha_n,\beta,\gamma\in\mathbb{C}$ satisfy
\begin{subequations}\label{Eq_AnSelberg-conditions}
\begin{gather}
\Re(\beta)>0, \qquad k_n \abs{\Re(\gamma)}<1,\qquad
\Re\big(\beta+(k_n-1)\gamma\big)>0, \\
\Re\big(\alpha_r+\cdots+\alpha_s+(r-s+i-1)\gamma\big)>0 \quad
\text{for $1\leq r\leq s\leq n$ and $1\leq i\leq k_r-k_{r-1}$}.
\end{gather}
\end{subequations}
Then the $\A_n$ Selberg integral refers to the integral evaluation
\begin{align}\label{Eq_AnSelberg}
&\Int_{C^{k_1,\dots,k_n}_{\gamma}[0,1]}
\prod_{r=1}^n\bigg(
\big\vert\Delta\big(\xar{r}\big)\big\vert^{2\gamma}
\prod_{i=1}^{k_r}\big(\xar{r}_i\big)^{\alpha_r-1}
\big(1-\xar{r}_i\big)^{\beta_r-1}\bigg)\\
&\qquad\qquad\qquad\qquad\qquad\times
\prod_{r=1}^{n-1}
\big\vert \Delta\big(\xar{r};\xar{r+1}\big)\big\vert^{-\gamma}\,
\dup\xar{1}\cdots\dup\xar{n} \notag \\
&\quad=\prod_{r=1}^n\prod_{i=1}^{k_r}
\frac{\Gamma(i\gamma)\Gamma(\beta_r+(i-k_{r+1}-1)\gamma)}
{\Gamma(\gamma)}\notag \\
&\qquad\times \prod_{1\leq r\leq s\leq n}\prod_{i=1}^{k_r-k_{r-1}}
\frac{\Gamma(\alpha_r+\cdots+\alpha_s+(r-s+i-1)\gamma)}
{\Gamma(\beta_s+\alpha_r+\cdots+\alpha_s+(k_s-k_{s+1}+i+r-s-2)\gamma)},\notag
\end{align}
where $\beta_1=\cdots=\beta_{n-1}=1$, $\beta_n:=\beta$ and $k_{n+1}:=0$.

The origin of the restrictions $\beta_1=\cdots=\beta_{n-1}=1$ and
$k_1\leq\dots\leq k_n$ is representation theoretic.
Let $\mathfrak{g}:=\mathfrak{sl}_{n+1}$, $\mathfrak{h}$ the Cartan subalgebra 
of $\mathfrak{g}$ and $\mathfrak{h}^{\ast}$ its dual.
For $I:=\{1,\dots,n\}$, let
$\{\alpha_i\}_{i\in I}\in\mathfrak{h}^{\ast}$,
$\{\omega_i\}_{i\in I}\in\mathfrak{h}^{\ast}$ and
$\{\alpha_i^{\vee}\}_{i\in I}\in\mathfrak{h}$
be the set of simple roots, fundamental weights and
simple coroots of $\mathfrak{g}$, so that 
$\langle \alpha_i^{\vee},\omega_j\rangle=\delta_{i,j}$.
Finally, let $P_{+}\subset\mathfrak{h}^{\ast}$ be the set of dominant 
integral weights, i.e., $\mu\in P_{+}$ if 
$\langle\mu,\alpha_i^{\vee}\rangle\in\mathbb{N}_0$ for all $i\in I$.
Now fix $\mu=\sum_{i\in I}(\mu_i-1)\omega_i\in P_{+}$ and 
$\nu=\sum_{i=I}(\nu_i-1)\omega_i\in P_{+}$
such that $\nu_1=\dots=\nu_{n-1}=1$,
and let $V_{\mu}$ and $V_{\nu}$ be two irreducible $\mathfrak{g}$-modules
of highest weight $\mu$ and $\nu$ respectively.
Then the following multiplicity-free tensor-product decomposition holds:
\[
V_{\mu}\otimes V_{\nu}=
\bigoplus_{\substack{0\leq k_1\leq\cdots\leq k_n\\[1pt] 
\mu+\nu-\sum_{i\in I} k_i\alpha_i\in P_{+}}}
V_{\mu+\nu-\sum_{i=1}^n k_i\alpha_i}.
\]
The $\alpha_1,\dots,\alpha_n$ and $\beta_n$ in 
\eqref{Eq_AnSelberg} are essentially continuous
analogues of $\mu_1,\dots,\mu_n$ and $\nu_n$, respectively,
and $\beta_i=\nu_i=1$ for all $1\leq i\leq n-1$.

The domain of integration $C_{\gamma}^{k_1,\dots,k_n}[0,1]$ in
\eqref{Eq_AnSelberg} takes the form of a $(k_1+\dots+k_n)$-dimensional chain.
Its precise form is not needed in this paper, and the interested reader is
referred to~\cite{ARW21,TV03,Warnaar08,Warnaar09} for details.
For $n=1$ the integration chain is independent of $\gamma$ and simplifies to
the $k$-simplex
\[
C_{\gamma}^k[0,1]=\big\{x\in\mathbb{R}^k:0<x_1<\dots<x_k<1\big\}.
\]
Up to a factor of $k!$, the $n=1$ case of \eqref{Eq_AnSelberg} is 
thus the original Selberg integral \eqref{Eq_Selberg}.
For $n=2$ the evaluation \eqref{Eq_AnSelberg} was first given by Tarasov and
Varchenko \cite{TV03}, and for general $n$ it is due to the third 
author~\cite{Warnaar09}.
There is also a finite field analogue of \eqref{Eq_AnSelberg} due to
Rim\'anyi and Varchenko \cite{RV21} which is not covered in our elliptic
generalisation below.

To state the elliptic $\A_n$ Selberg integral we introduce some further
notation.
For $z=(z_1,\dots,z_k)\in(\mathbb{C}^{\ast})^k$, 
$w=(w_1,\dots,w_\ell)\in(\mathbb{C}^{\ast})^{\ell}$ and
$c\in\mathbb{C}^{\ast}$, define
\begin{equation}\label{Eq_DeltaSe}
\DeltaSe(z;w;c;p,q):=
\prod_{i=1}^k\prod_{j=1}^{\ell}\Gammapq\big(cz_i^{\pm}w_j^{\pm}\big).
\end{equation}
Whereas the elliptic Selberg density \eqref{Eq_Selberg-density}
should be viewed as the elliptic analogue of the integrand of the Selberg
integral \eqref{Eq_Selberg}, the above function for $c=(pq/t)^{1/2}$
plays the role of $\abs{\Delta(x;y)}^{-\gamma}$ in the elliptic analogue of
\eqref{Eq_AnSelberg}.
This same special case of \eqref{Eq_DeltaSe} previously appeared in the study 
of elliptic integrable systems, see e.g., \cite{AN22,KNS09,Ruijsenaars09}
and, as shown in \cite{KNS09,Ruijsenaars09}, satisfies a remarkable duality 
with respect to the $8$-parameter van Diejen difference 
operator \cite{vDiejen94}.

We now combine the two elliptic Selberg densities to form the
$\A_n$ elliptic Selberg density
\begin{align}\label{Eq_An-S-density}
&\DeltaS\big(\zar{1},\dots,\zar{n};t_1,\dots,t_{2n+4};c;t;p,q\big)\\
&\quad:=\prod_{r=1}^{n-1}\Big(\DeltaSv\big(\zar{r};
c^{r-n}t_{2r-1},c^{r-n}t_{2r},
tc^{n-r}/t_{2r+1},tc^{n-r}/t_{2r+2};t;p,q\big) \notag \\
&\qquad\qquad\quad\times
\DeltaSe\big(\zar{r};\zar{r+1};c;p,q\big)\Big) \notag \\
&\qquad\times
\DeltaSv\big(\zar{n};
t_{2n-1},t_{2n},t_{2n+1},t_{2n+2},t_{2n+3},t_{2n+4};t;p,q\big), \notag
\end{align}
where $\zar{r}=\big(\zar{r}_1,\dots,\zar{r}_{k_r}\big)$.
Suppressing the dependence on $c,t,t_1,\dots,t_{2n+4},p,q$, the individual 
densities making up the $\A_n$ density should be thought of as corresponding
to the vertices and edges of the $\A_n$ Dynkin diagram as follows:
\medskip
\begin{center}
\begin{tikzpicture}[scale=1]
\begin{scope}[draw=blue]
\draw (0,0)--(6.5,0); \draw (8.5,0)--(12,0);
\draw[dashed] (6.5,0)--(8.5,0);
\draw[fill=white] (0,0) circle (0.21cm);
\draw[fill=white] (3,0) circle (0.21cm); 
\draw[fill=white] (6,0) circle (0.21cm);
\draw[fill=white] (9,0) circle (0.21cm);
\draw[fill=white] (12,0) circle (0.21cm);
\end{scope}
\draw (0,0) node {$\scriptscriptstyle 1$};
\draw (3,0) node {$\scriptscriptstyle 2$};
\draw (6,0) node {$\scriptscriptstyle 3$};
\draw (9,0) node {$\scriptscriptstyle n\!-\!1$};
\draw (12,0) node {$\scriptscriptstyle n$};
\draw (0,0.6) node {$\scriptstyle\DeltaSv(\zar{1})$};
\draw (3,0.6) node {$\scriptstyle\DeltaSv(\zar{2})$};
\draw (6,0.6) node {$\scriptstyle\DeltaSv(\zar{3})$};
\draw (9,0.6) node {$\scriptstyle\DeltaSv(\zar{n-1})$};
\draw (12,0.6) node {$\scriptstyle\DeltaSv(\zar{n})$};
\draw (1.5,-0.4) node {$\scriptstyle\DeltaSe(\zar{1},\zar{2})$};
\draw (4.5,-0.4) node {$\scriptstyle\DeltaSe(\zar{2},\zar{3})$};
\draw (10.5,-0.4) node {$\scriptstyle\DeltaSe(\zar{n-1},\zar{n})$};
\end{tikzpicture}
\end{center}
Finally, for $z=(z_1,\dots,z_k)$, we let
$\frac{\dup z}{z}:=\frac{\dup z_1}{z_1}\cdots\frac{\dup z_k}{z_k}$.

\begin{theorem}[$\A_n$ elliptic Selberg integral]\label{Thm_E-AnSelberg}
Let $n$ be a positive integer and $k_1,\dots,k_n$ integers such that
$0=:k_0\leq k_1 \leq\cdots \leq k_n$. 
For $p,q,t\in\mathbb{C}^{\ast}$ such that $\abs{p},\abs{q},\abs{t},\abs{pq/t}<1$,
fix a branch of $c:=(pq/t)^{1/2}$, and let
$t_1,\dots,t_{2n+4}\in\mathbb{C}^{\ast}$ such that the balancing condition 
\begin{equation}\label{Eq_An-balancing}
t^{k_r-k_{r-1}+k_n-2}t_{2r-1}t_{2r}t_{2n+1}t_{2n+2}t_{2n+3}t_{2n+4}=pq
\end{equation}
holds for all $1\leq r\leq n$.
Then
\begin{align}\label{Eq_E-AnSelberg}
&\int \DeltaS\big(\zar{1},\dots,\zar{n};t_1,\dots,t_{2n+4};c;t;p,q\big)\,
\frac{\dup \zar{1}}{\zar{1}}\cdots\frac{\dup \zar{n}}{\zar{n}} \\
&\quad= \prod_{r=1}^n \prod_{i=1}^{k_r-k_{r-1}}
\Gammapq(t^i,t^{i-1}c^{2r-2n}t_{2r-1}t_{2r})
\prod_{2n+1\leq r<s\leq 2n+4}\,\prod_{i=1}^{k_n}
\Gammapq(t^{i-1}t_rt_s) \notag \\[1mm]
&\qquad\times \prod_{1\leq r<s\leq n} \prod_{i=1}^{k_r-k_{r-1}}
\Gammapq(t^it_{2r-1}/t_{2s-1},t^it_{2r}/t_{2s-1},
t^it_{2r-1}/t_{2s},t^it_{2r}/t_{2s}) \notag \\[1mm] 
&\qquad\times
\prod_{r=1}^n\prod_{s=2n+1}^{2n+4}\prod_{i=1}^{k_r-k_{r-1}}
\Gammapq(t^{i-1}t_{2r-1}t_s,t^{i-1}t_{2r}t_s), \notag
\end{align}
where $\zar{r}=\big(\zar{r}_1,\dots,\zar{r}_{k_r}\big)$ for all $1\leq r\leq n$.
\end{theorem}

The $(k_1+\dots+k_n)$-dimensional contour of integration of the 
$\A_n$ Selberg integral has the product structure
\[
\underbrace{C_1\times\cdots\times C_1}_{\text{$k_1$-times}}\times 
\underbrace{C_2\times\cdots\times C_2}_{\text{$k_2$-times}}\times\cdots\cdots\times
\underbrace{C_n\times\cdots\times C_n}_{\text{$k_n$-times}},
\]
where $C_r$ for each $1\leq r\leq n$ is a positively oriented smooth
Jordan curve around $0$ such that $C_r=C_r^{-1}$.
Moreover, for $1\leq r\leq n-1$, the elements of the sets
\begin{subequations}
\begin{equation}
\label{Eq_in1}
c^{r-n} t_{2r+s-2} p^{\mathbb{N}_0} q^{\mathbb{N}_0},\quad
tc^{n-r} t_{2r+s}^{-1} p^{\mathbb{N}_0} q^{\mathbb{N}_0}
\; (1\leq s\leq 2),\quad
tp^{\mathbb{N}_0} q^{\mathbb{N}_0}C_r,\quad
cp^{\mathbb{N}_0} q^{\mathbb{N}_0}C_{r\pm 1}
\end{equation}
all lie in the interior of $C_r$, and the elements of
\begin{equation}\label{Eq_in2}
t_{s+2n-2} p^{\mathbb{N}_0} q^{\mathbb{N}_0} \; (1\leq s\leq 6),\quad
tp^{\mathbb{N}_0} q^{\mathbb{N}_0}C_n,\quad
cp^{\mathbb{N}_0} q^{\mathbb{N}_0}C_{n-1}
\end{equation}
\end{subequations}
all lie in the interior of $C_n$, where $C_0:=0$.
These conditions on the $C_r$ in particular imply that 
$c^2C_r\in\textrm{int}(C_r)$ for $2\leq r\leq n$, explaining why
$\abs{c^2}=\abs{pq/t}<1$.
For $n=1$ this restriction can obviously be dropped.
Furthermore, for $n=1$ the balancing condition \eqref{Eq_An-balancing}
simplifies to $t^{2k_1-2}t_1t_2\cdots t_6=pq$.
Taking $\abs{t_1},\dots,\abs{t_6}<1$ it then follows that \eqref{Eq_in2}
is satisfied for $C=\mathbb{T}$, so that the integral reduces to
\eqref{Eq_ESelberg}.
For $n\geq 2$ it is generally not possible to restrict the parameters
such that $C_r=\mathbb{T}$ for all $1\leq r\leq n$.
For example, if $C_r=\mathbb{T}$ for all $r$, it follows from \eqref{Eq_in1}
that $c^{r-n} t_{2r-1}, c^{r-n} t_{2r},tc^{n-r} t_{2r+1}^{-1},
tc^{n-r} t_{2r+2}^{-1}$ all lie in the interior of $\mathbb{T}$.
By \eqref{Eq_An-balancing} and $\abs{t}<1$ this would impose the
condition that $k_{r+1}-2k_r+k_{r-1}\geq -1$ for all $1\leq r\leq n-1$.

All of the integral formulas listed thus far admit generalisations in which 
the integrand is multiplied by an appropriate symmetric function or 
$\mathrm{BC}_n$-symmetric function.
In the case of \eqref{Eq_Selberg} the most general such integral was 
discovered by Alba, Fateev, Litvinov and Tarnopolsky (AFLT) \cite{AFLT11}
and contains a pair of Jack polynomials in the integrand.
The AFLT integral includes the well-known Kadell integral \cite{Kadell97} 
(which contains one Jack polynomial) 
and the Hua--Kadell integral \cite{Hua63,Kadell93}
(which contains two Jack polynomials but assumes $\beta=\gamma$) 
as special cases.
In our previous paper \cite{ARW21} the AFLT integral was extended to the
elliptic case, as well as to $\A_n$.
In Section~\ref{Sec_AFLT-proof} we unify both these results by proving an
elliptic $\A_n$ AFLT integral.
In this integral the Jack polynomials in the integrand of the non-elliptic
$\A_n$ AFLT integral are replaced by a pair of elliptic interpolation
functions~\cite{Rains12}.
Our approach to the elliptic $\A_n$ Selberg and AFLT integrals is based on
a recursion for a generalisation of the elliptic interpolation
functions, known as the elliptic interpolation kernel~\cite{Rains18}.
This differs from the approaches taken in \cite{ARW21}, where the
non-elliptic $\A_n$ AFLT integral is proved using Cauchy-type
identities for Macdonald polynomials and the $\A_1$ elliptic AFLT 
integral is proved using known integral identities for elliptic 
interpolation functions.

\medskip

The remainder of the paper is organised as follows.
In the next section we review some standard definitions and notation from
the theory of elliptic beta integrals.
Section~\ref{Sec_elliptic-funcs} is devoted to several classes of elliptic 
special functions, including the elliptic interpolation functions and the 
elliptic interpolation kernel.
The latter forms the basis of our approach to Theorem~\ref{Thm_E-AnSelberg}.
In Section~\ref{Sec_AFLT-proof} we first discuss the original AFLT integral
and its $\A_n$ analogue, and then state and prove an elliptic 
$\A_n$ AFLT integral.
As a special case this yields Theorem~\ref{Thm_E-AnSelberg}.

\section{Elliptic preliminaries}
Throughout this paper we assume that $p,q\in\mathbb{C^{\ast}}$ such that 
$\abs{p},\abs{q}<1$.

\subsection{Partitions}
A partition $\la=(\la_1,\la_2,\dots)$ is a weakly decreasing sequence
of nonnegative integers
$\la_i$ such that only finitely many $\la_i$ are nonzero.
The nonzero $\la_i$ are called the parts of $\la$, 
and the number of parts is the length of $\la$, denoted by $l(\la)$.
Partitions are identified up to the number of trailing zeroes, so that, 
for example, $(3,1,1)=(3,1,1,0,\dots)$.
We write $\Part$ for the set of all partitions and $\Part_n$ for
the set of all partitions of length at most $n$.
In particular, $\Part_0=\{0\}$, with $0$ the unique partition of $0$.
If the sum of the parts, denoted $\abs{\la}$, is equal to some integer $n$,
then $\la$ is said to be a partition of $n$, which is also
written $\la\vdash n$.
If $\la$ is a partition, we write $(i,j)\in\la$ to mean any pair
of integers $(i,j)$ such that $1\leq i\leq l(\la)$ and $1\leq j\leq\la_i$.
If $\la$ is a partition, its conjugate $\la'$ is defined by
$\la_i':=\abs{\{j\in \mathbb{N}:\la_j\geq i\}}$.
For example $(7,4,2,1,1)'=(5,3,2,2,1,1,1)$. 
For a pair of partitions $\la,\mu$ we write $\mu\subseteq\la$ if
$\mu_i\leq\la_i$ for all $i$.
If $\la,\mu$ further satisfy 
$\la_1\geq\mu_1\geq\la_2\geq\mu_2\geq\cdots$
(i.e., $\mu\subseteq\la$ and $\la'_i-\mu'_i\in\{0,1\}$ for all $i\geq 1$),
we write $\mu\prec\la$.
(In this case the skew shape $\la/\mu$ is known as a horizontal strip.)
 
We refer to elements of $\Part^2$ as bipartitions, and to distinguish
partitions from bipartitions a bold font such as $\bla$ is used for the
latter.
In particular, $\bzero$ denotes the bipartition $(0,0)$.
If $\bla=(\lar{1},\lar{2})$ and $\bmu=(\mur{1},\mur{2})$ are
bipartitions then the notation $\bmu\subseteq\bla$ is shorthand for the
termwise inclusions $\mur{1}\subseteq\lar{1}$ and
$\mur{2}\subseteq\lar{2}$.
The notation $\bmu\prec\bla$ is similarly defined.
For $\bla\in\Part_n^2$, the spectral vector $\spec{\bla}_{n;t;p,q}$ is
given by
\[
\spec{\bla}_{n;t;p,q}:=
\Big(p^{\lar{1}_1}q^{\lar{2}_1}t^{n-1},p^{\lar{1}_2}q^{\lar{2}_2}t^{n-2},
\dots,p^{\lar{1}_{n-1}}q^{\lar{2}_{n-1}}t,p^{\lar{1}_n}q^{\lar{2}_n}\Big),
\]
so that
\[
\Spec{\big(\lar{1},\lar{2}\big)}_{n;t;p,q}=
\Spec{\big(\lar{2},\lar{1}\big)}_{n;t;q,p}.
\] 

\subsection{Elliptic preliminaries}\label{Sec_elliptic}

A key ingredient in the theory of elliptic hypergeometric functions
is the modified theta function, defined as
\[
\thetap(z):=(z;p)_{\infty}(p/z;p)_{\infty},
\]
for $z\in\mathbb{C}^{\ast}$.
This function is quasi periodic along annuli
\begin{equation}\label{Eq_quasi}
\thetap(pz)=-z^{-1}\thetap(z),
\end{equation}
satisfies the symmetry $\thetap(z)=-z\,\thetap(1/z)$,
and features in the functional equation
\begin{equation}\label{Eq_functionalEq}
\Gammapq(pz)=\thetaq(z)\Gammapq(z)
\end{equation}
for the elliptic gamma function.

For $n$ an integer, the elliptic shifted factorial is defined as
\begin{equation}\label{Eq_E-fac}
(z;q,p)_n:=\frac{\Gammapq(q^nz)}{\Gammapq(z)},
\end{equation}
where it is noted that for $n\geq 0$,
\[
(z;q,p)_n=\prod_{i=1}^n\thetap(zq^{i-1}).
\]
The elliptic shifted factorial has three important generalisations
to partitions, given by
\begin{align*}
C^0_{\la}(z;q,t;p)&:=
\prod_{(i,j)\in\la}\thetap\big(zq^{j-1}t^{1-i}\big), \\
C^{+}_{\la}(z;q,t;p)&:=
\prod_{(i,j)\in\la}\thetap\big(zq^{\la_i+j-1}t^{2-\la_j'-i}\big),\\
C^{-}_{\la}(z;q,t;p)&:=
\prod_{(i,j)\in\la}\thetap\big(zq^{\la_i-j}t^{\la_j'-i}\big).
\end{align*}
Note that $C_{\la}^0(z;q,t;p)$ is sometimes denoted $(z;q,t;p)_{\la}$
in the literature on elliptic hypergeometric series.

For all of the functions defined above, condensed notation such as
\[
C^0_{\la}(z_1,\dots,z_k;q,t;p):=
C^0_{\la}(z_1;q,t;p)\cdots C^0_{\la}(z_n;q,t;p)
\]
will be employed.
As further shorthand notation we define the following well-poised ratio
of products of elliptic shifted factorials:
\[
\Delta_{\la}^0(a\vert b_1,\dots,b_n;q,t;p):=
\prod_{i=1}^n\frac{C_\la^0(b_i;q,t;p)}{C_\la^0(pqa/b_i;q,t;p)},
\]
which satisfies the reflection equation
\begin{equation}\label{Eq_Delta-symm}
\Delta_{\la}^0(a\vert b_1,\dots,b_n;q,t;p)=
\frac{1}{\Delta_{\la}^0(a\vert pqa/b_1,\dots,pqa/b_n;q,t;p)}.
\end{equation}

To preserve $p,q$-symmetry in many of the elliptic functions and integrals
considered in this paper, we require an extension of the above definitions
to bipartitions, and for any function
$f_{\la}(a_1,\dots,a_n;q,t;p)$ or $f_{\la/\mu}(a_1,\dots,a_n;q,t;p)$
we define
\begin{subequations}
\begin{align}\label{Eq_bi-extend}
f_{\bla}(a_1,\dots,a_n;t;p,q)&:=
f_{\lar{1}}(a_1,\dots,a_n;p,t;q)
f_{\lar{2}}(a_1,\dots,a_n;q,t;p), \\[1mm]
f_{\bla/\bmu}(a_1,\dots,a_n;t;p,q)&:=
f_{\lar{1}/\mur{1}}(a_1,\dots,a_n;p,t;q) 
f_{\lar{2}/\mur{2}}(a_1,\dots,a_n;q,t;p).
\end{align}
\end{subequations}
Interchanging $p$ and $q$ is thus the same as interchanging the 
two components of $\bla$ and, in the skew case, the two components of $\bmu$.
By \eqref{Eq_reflection} and \eqref{Eq_E-fac} followed by the
use of the quasi periodicity \eqref{Eq_quasi}, it may be shown that
\begin{equation}\label{Eq_Gamma-Delta}
\prod_{i=1}^n \frac{\Gammapq\big(at^{1-i}p^{\lar{1}_i}q^{\lar{2}_i},
bt^{i-1}p^{-\lar{1}_i}q^{-\lar{2}_i}\big)}
{\Gammapq(at^{1-i},bt^{i-1})}=
\Big(\frac{pq}{ab}\Big)^{\sum_{i=1}^n \lar{1}_i\lar{2}_i}\,
\Delta^0_{\bla}(a/b\vert a;t;p,q),
\end{equation}
for $\bla\in\Part_n^2$.

\subsection{The Dixon and Selberg densities}

In addition to the elliptic Selberg density \eqref{Eq_Selberg-density},
we need the elliptic Dixon density
\[
\DeltaD(z_1,\dots,z_k;t_1,\dots,t_m;p,q)
:=\varkappa_k\prod_{1\leq i<j\leq k}\frac{1}{\Gammapq(z_i^{\pm}z_j^{\pm})}
\prod_{i=1}^k\frac{\prod_{r=1}^m\Gammapq(t_rz_i^{\pm})}
{\Gammapq(z_i^{\pm 2})},
\]
with $\varkappa_k$ given in~\eqref{Eq_kappa-n}.
This is related to the Selberg density by
\begin{align*}
&\DeltaSv(z_1,\dots,z_k;t_1,\dots,t_m;t;p,q) \\
&\quad=\DeltaD(z_1,\dots,z_k;t_1,\dots,t_m;t;p,q)\,
\Gammapq^k(t)\prod_{1\leq i<j\leq k}\Gammapq(tz_i^\pm z_j^\pm).
\end{align*}
Apart from possible balancing conditions, or restrictions to certain
subsets of the complex plane, it will be assumed throughout this paper
that parameters such as $t_1,\dots,t_m,p,q,t$ are in generic position.

We say that a function $f:(\mathbb{C}^{\ast})^k\longrightarrow\mathbb{C}$
is $\mathrm{BC}_k$-symmetric if $f(x_1,\dots,x_k)$ is invariant under the
natural action of the hyperoctahedral group
$\Symm_k\ltimes(\mathbb{Z}/2\mathbb{Z})^k$.
For $f$ a $\mathrm{BC}_k$-symmetric meromorphic function and
$t,t_1,\dots,t_6\in\mathbb{C}^{\ast}$ such that $\abs{t}<1$, 
we define the Selberg average of $f$ as
\begin{equation}\label{Eq_A1-average}
\big\langle f\big\rangle_{t_1,\dots,t_6;t;p,q}^k
:=\frac{1}{S_k(t_1,\dots,t_6;t;p,q)}\Int \! f(z)
\DeltaSv(z;t_1,\dots,t_6;t;p,q)\,\frac{\dup z}{z},
\end{equation}
where $S_k(t_1,\dots,t_6;t;p,q)$ denotes the $\A_1$ elliptic Selberg 
integral \eqref{Eq_ESelberg} and where it is assumed that
$t^{2k-2}t_1\cdots t_6=pq$.
The contour of the integral on the right has the form $C^k$, where
$C=C^{-1}$ is a positively oriented smooth Jordan curve around $0$ such
that
\[
t_r p^{\mathbb{N}_0} q^{\mathbb{N}_0} \; (1\leq r\leq 6),\quad
tp^{\mathbb{N}_0} q^{\mathbb{N}_0}C,
\]
as well as any sequence of poles of $f$ tending to zero, excluding those 
cancelled by the univariate part of the Dixon density, all lie in the
interior of $C$.
If $f$ is analytic on $(\mathbb{C}^{\ast})^k$ and
$\abs{t_1},\dots,\abs{t_6}<1$, we may take $C=\mathbb{T}$.

\section{Elliptic interpolation functions and the interpolation kernel}
\label{Sec_elliptic-funcs}

The purpose of this section is to introduce the $\mathrm{BC}_k$-symmetric
elliptic interpolation functions and the closely related interpolation
kernel.
The interpolation functions will play the role of Jack polynomials in our
elliptic analogues of the $\A_n$ AFLT, Kadell and Hua--Kadell integrals.
The interpolation kernel is a crucial ingredient in our proof of the
various elliptic $\A_n$ Selberg integrals, allowing us to establish a
recursion in the rank $n$.

\subsection{Elliptic interpolation functions}\label{Sec_EIF}
Below we give a brief review of the elliptic interpolation functions.
The reader may consult \cite{CG06,Rains06,Rains10,Rains12,Rains18,RW20}
for more complete accounts.

For $\bmu\in\Part_k^2$, $x=(x_1,\dots,x_k)\in(\mathbb{C}^{\ast})^k$ and
$a,b,t\in\mathbb{C}^{\ast}$, the $\mathrm{BC}_k$-symmetric elliptic
interpolation function is denoted by
\[
\Ri_{\bmu}(x;a,b;t;p,q),
\]
and consists of a $q$-elliptic factor and $p$-elliptic factor:
\[
\Ri_{\bmu}(x;a,b;t;p,q)=\Ri_{\mur{1}}(x;a,b;p,t;q)
\Ri_{\mur{2}}(x;a,b;q,t;p).
\]
As usual in symmetric function theory, $\Ri_0(x;a,b;q,t;p)=1$.
We also adopt the convention that 
$\Ri_{\bmu}(x;a,b;t;p,q)=0$ if $\bmu$ is a bipartition such that
$\bmu\not\in\Part_k^2$, i.e., if the length of at least one of 
$\mur{1},\mur{2}$ exceeds $k$.

The fundamental property of the elliptic interpolation functions is 
the vanishing
\[
\Ri_{\bmu}(a\spec{\bla}_{k;t;p,q};a,b;t;p,q)=0
\]
for all $\bla\in\Part_k^2$ such that $\bmu\not\subseteq\bla$.
The $\mathrm{BC}_k$-symmetric interpolation function
$\Ri_{\mu}(x;a,b;q,t;p)$ generalises Okounkov's $\mathrm{BC}_k$-symmetric
interpolation Macdonald polynomial $P^{\ast}_{\mu}(x;q,t,s)$, which
satisfies a similar vanishing property and contains the ordinary Macdonald
polynomial $P_{\mu}(x;q,t)$ as its top-homogeneous degree component; 
see \cite{Okounkov98,Rains05} for details.
The interpolation functions completely factorise under principal 
specialisation:
\[
\Ri_{\bmu}(v\spec{\bzero}_{k;t;p,q};a,b;t;p,q)=
\Delta^0_{\bmu}(t^{k-1}a/b\vert t^{k-1}av,a/v;t;p,q).
\]
If the parameters satisfy $t^kab=pq$ then the interpolation functions
are said to be of Cauchy type and once again factorise:
\begin{equation}\label{Eq_Cauchy-type}
\Ri_{\bmu}(x;a,b;t;p,q)=
\Delta^0_{\bmu}\big(t^{k-1}a/b\big\vert 
t^{k-1}ax_1^{\pm},\dots,t^{k-1}ax_k^{\pm};t;p,q\big).
\end{equation}

The elliptic binomial coefficients 
\[
\obinomE{\bla}{\bmu}_{[a,b];t;p,q}
=\obinomE{\lar{1}}{\mur{1}}_{[a,b];p,t;q}
\obinomE{\lar{2}}{\mur{2}}_{[a,b];q,t;p}
\]
are defined as normalised connection coefficients between 
the elliptic $\mathrm{BC}_k$ interpolation functions:
\begin{align}\label{Eq_connection-coeff}
&\Ri_{\bla}(x;a,b;t;p,q) \\
&\quad=\sum_{\bmu} \obinomE{\bla}{\bmu}_{[t^{k-1}a/b,a/a'];t;p,q}
\frac{\Delta^0_{\bla}(t^{k-1}a/b\vert t^{k-1}aa';t;p,q)}
{\Delta^0_{\bmu}(t^{k-1}a'/b\vert t^{k-1}aa';t;p,q)}\,
\Ri_{\bmu}(x;a',b;t;p,q). \notag 
\end{align}
It may be shown that this definition is independent of the choice of $k$
and that $\obinomE{\bla}{\bmu}_{[a,b];t;p,q}$ vanishes unless
$\bmu\subseteq\bla$. 
Moreover, for $b=t$ there is additional vanishing and
\begin{equation}\label{Eq_verticalstrip}
\obinomE{\bla}{\bmu}_{[a,t];t;p,q}=0 \quad
\text{unless $\bmu\prec\bla$.}
\end{equation}
For notational purposes it is convenient to extending the definition
of the elliptic binomials to
\begin{equation}\label{Eq_Ebinom-v}
\obinomE{\bla}{\bmu}_{[a,b](v_1,\dots,v_k);t;p,q}:=
\frac{\Delta^0_{\bla}(a\vert v_1,\dots,v_k;t;p,q)}
{\Delta^0_{\bmu}(a/b\vert v_1,\dots,v_k;t;p,q)}\,
\obinomE{\bla}{\bmu}_{[a,b];t;p,q}.
\end{equation}
By \eqref{Eq_Delta-symm},
\begin{equation}\label{Eq_Ebinom-reduction}
\obinomE{\bla}{\bmu}_{[a,b](v_1,\dots,v_k,w,pqa/bw);t;p,q}=
\frac{\Delta^0_{\bla}(a\vert w;t;p,q)}
{\Delta^0_{\bla}(a\vert bw;t;p,q)}\,
\obinomE{\bla}{\bmu}_{[a,b](v_1,\dots,v_k);t;p,q}.
\end{equation}
The reader is warned that the elliptic binomial coefficients for
$\bmu=0$ or $\bmu=\bla$ do not simplify to $1$:
\begin{equation}\label{Eq_simplecases}
\obinomE{\bla}{\bzero}_{[a,b];t;p,q}=
\Delta^0_{\bla}(a\vert b;t,p,q)\quad\text{and}\quad
\obinomE{\bla}{\bla}_{[a,b];t;p,q}=
\frac{C^+_{\bla}(a;t;p,q)}{C^+_{\bla}(a/b;t;p,q)}.
\end{equation}
For $b=1$ they trivialise to
\begin{equation}\label{Eq_b=1}
\obinomE{\bla}{\bmu}_{[a,1](v_1,\dots,v_k);t;p,q}=\delta_{\bla\bmu},
\end{equation}
as follows immediately from the definition \eqref{Eq_connection-coeff}.
The elliptic binomial coefficients satisfy an analogue of the elliptic
Jackson sum as follows \cite[Theorem~4.1]{Rains06}
(see also \cite[Equation~(3.7)]{CG06}):
\begin{equation}\label{Eq_Cn-Jackson}
\sum_{\bmu}\Delta^0_{\bmu}(a/b\vert d,e;t;p,q)\,
\obinomE{\bla}{\bmu}_{[a,b];t;p,q}\obinomE{\bmu}{\bnu}_{[a/b,c/b];t;p,q}
=\obinomE{\bla}{\bnu}_{[a,c](bd,be);t;p,q},
\end{equation}
where $bcde=apq$.

Using \eqref{Eq_Ebinom-v}, the connection coefficient formula may be
written more succinctly as
\begin{equation}\label{Eq_connection}
\Ri_{\bla}(x;a,b;t;p,q)
=\sum_{\bmu}\obinomE{\bla}{\bmu}_{[t^{k-1}a/b,a/a'](t^{k-1}aa');t;p,q}
\Ri_{\bmu}(x;a',b;t;p,q).
\end{equation}
Choosing $a'=pq/t^k b$ and using \eqref{Eq_Cauchy-type} implies that
\begin{align}\label{Eq_connection-spec}
&\Ri_{\bla}(x;a,b;t;p,q) \\ 
&\quad=\sum_{\bmu}
\obinomE{\bla}{\bmu}_{[t^{k-1}a/b,t^kab/pq](pqa/tb);t;p,q}
\Delta^0_{\bmu}\big(pq/tb^2\big\vert
pqx_1^{\pm}/tb,\dots,pqx_k^{\pm}/tb;t;p,q\big). \notag
\end{align}

The elliptic binomial coefficients may be used to define suitable
skew analogues of the interpolation functions, and for arbitrary
$\bla,\bnu\in\Part^2$ and $k$ a nonnegative integer,
\begin{align}\label{Eq_Riskew}
&\Ri_{\bla/\bnu}([v_1,\dots,v_{2k}];a,b;t;p,q) \\
&\quad:=\sum_{\bmu}
\Delta^0_{\bmu}(pq/b^2\vert pq/bv_1,\dots,pq/bv_{2k};t;p,q)\,
\obinomE{\bla}{\bmu}_{[a/b,ab/pq];t;p,q}
\obinomE{\bmu}{\bnu}_{[pq/b^2,pqV/ab];t;p,q}, \notag
\end{align}
where $a,b,t,v_1,\dots,v_{2k}\in\mathbb{C}^{\ast}$ and
$V:=v_1\cdots v_{2k}$.
Obviously, we have vanishing unless $\bnu\subseteq\bla$.
It follows from the definition that the skew interpolation functions
are $\mathfrak{S}_{2k}$-symmetric functions, rather than
$\mathrm{BC}_k$-symmetric.
As explained in more detail in \cite{ARW21,Rains12}, the use of the
brackets around the $v_1,\dots,v_{2k}$ is a reflection of the close
connection with plethystic notation, see for example, 
\cite[Equation (6.7)]{ARW21}.
Taking $\bnu=\bzero$ in \eqref{Eq_Riskew} and using \eqref{Eq_simplecases}
it follows that 
\[
\Ri_{\bla/\bzero}([v_1,\dots,v_{2k}];a,b;t;p,q)
\]
is symmetric in $v_1,\dots,v_{2k},a/V$.\label{page_symm}

The definition of the skew interpolation functions combined with the 
elliptic Jackson summation \eqref{Eq_Cn-Jackson} implies the branching rule
\begin{align}\label{Eq_R-branching}
&\Ri_{\bla/\bnu}([v_1,\dots,v_{2k},w_1,w_2];a,b;t;p,q)\\
&\quad=\sum_{\bmu}
\obinomE{\bla}{\bmu}_{[a/b,w_1w_2](a/w_1,a/w_2);t;p,q}
\Ri_{\bmu/\bnu}([v_1,\dots,v_{2k}];a/w_1w_2,b;t;p,q).\notag
\end{align}
Taking $w_1w_2=1$ and using \eqref{Eq_b=1}, this shows that
\begin{equation}\label{Eq_v1v2=1}
\Ri_{\bla/\bmu}([v_1,\dots,v_{2k}];a,b;t;p,q)|_{v_{2k-1}v_{2k}=1}=
\Ri_{\bla/\bmu}([v_1,\dots,v_{2k-2}];a,b;t;p,q).
\end{equation}
By symmetry this extends to any pair of variables whose product is $1$.
Similarly, from \eqref{Eq_Cn-Jackson} with $c=1$ (so that
$\Delta^0_{\bmu}(a/b\vert d,e;t,p,q)=1$) and \eqref{Eq_b=1}, it follows that
\begin{equation}\label{Eq_novariables}
\Ri_{\bla/\bmu}([\;];a,b;t;p,q)=\delta_{\bla\bmu}.
\end{equation}
By \eqref{Eq_R-branching} with $k=0$ this generalises to 
\[
\Ri_{\bla/\bmu}([v_1,v_2];a,b;t;p,q)
=\obinomE{\bla}{\bmu}_{[a/b,v_1v_2](a/v_1,a/v_2);t;p,q}.
\]

Let $v_{2i-1}v_{2i}=t$ for all $1\leq i\leq k$. 
Then, by \eqref{Eq_verticalstrip}, \eqref{Eq_R-branching},
\eqref{Eq_novariables} and induction on $k$, it follows that 
$\Ri_{\bla/\bmu}([v_1,\dots,v_{2k}];a,b;t;p,q)$ vanishes unless
there exist $\bkappa^{(1)},\dots,\bkappa^{(k)}\in\Part^2$ such that
$\bmu\prec\bkappa^{(1)}\prec\cdots\prec\bkappa^{(k)}\prec\bla$.
In particular, for $\bmu=\bzero$ we have vanishing if 
$\bla\not\in \Part_k^2$.

A further consequence of \eqref{Eq_b=1} is that for $ab=pq$
\[
\Ri_{\bla/\bmu}([v_1,\dots,v_{2k}];a,b;t;p,q)
=\Delta^0_{\bla}(a/b\vert a/v_1,\dots,a/v_{2k};t;p,q)\,
\obinomE{\bla}{\bmu}_{[a/b,V];t;p,q},
\]
so that in particular for $ab=pq$,
\begin{equation}\label{Eq_skew-factors}
\Ri_{\bla/\bzero}([v_1,\dots,v_{2k}];a,b;t;p,q)
=\Delta^0_{\bla}(a/b\vert a/v_1,\dots,a/v_{2k},V;t;p,q).
\end{equation}

Specialising $(v_1,\dots,v_{2k};\bnu)$ to
$(x_1,x_1^{-1},\dots,x_k,x_k^{-1};\bzero)$ in \eqref{Eq_Riskew}, using
\eqref{Eq_simplecases}, and then comparing the resulting equation
with \eqref{Eq_connection-spec} yields the nonvanishing case (i.e.,
$\bla\in\Part_k^2$) of
\begin{align*}
&\Ri_{\bla/\bzero}\big([t^{1/2}x_1^{\pm},\dots,t^{1/2}x_k^{\pm}];
t^{k-1/2}a,t^{1/2}b;t;p,q\big) \\
&\quad=\Delta^0_{\bla}(t^{k-1}a/b\vert t^k;t;p,q)
\Ri_{\bla}(x_1,\dots,x_k;a,b;t;p,q).
\end{align*}
The above identity shows that, up to simple factor, the non-skew elliptic
interpolation functions are a special instances of the skew interpolation
functions.

For our purposes it will be convenient to define the hybrid interpolation
function
\begin{align}\label{Eq_Rinew}
& \Ri_{\bmu}(x_1,\dots,x_k;v_1,\dots,v_{2\ell};a,b;t;p,q) \\
&\quad:= \frac{\Ri_{\bmu/\bzero}
\big([t^{1/2}x_1^{\pm},\dots,t^{1/2}x_k^{\pm},
t^{1/2}v_1,\dots,t^{1/2}v_{2\ell}];t^{k-1/2}a,t^{1/2}b;t;p,q\big)}
{\Delta^0_{\bmu}(t^{k-1}a/b\vert t^{k+\ell}v_1\cdots v_{2\ell};t;p,q)}
\notag
\end{align}
for arbitrary $\bmu\in\Part^2$, so that
\[
\Ri_{\bmu}(x_1,\dots,x_k;\text{--}\,;a,b;t;p,q)=
\Ri_{\bmu}(x_1,\dots,x_k;a,b;t;p,q).
\]
By \eqref{Eq_v1v2=1},
\[
\Ri_{\bmu}(x_1,\dots,x_k;v_1,\dots,v_{2\ell};a,b;t;p,q)\vert_
{v_{2\ell-1}v_{2\ell}=1/t}=
\Ri_{\bmu}(x_1,\dots,x_k;v_1,\dots,v_{2\ell-2};a,b;t;p,q),
\]
and, from the definition,
\begin{align*}
&\Ri_{\bmu}(x_1,\dots,x_k;v_1,\dots,v_{2\ell};a,b;t;p,q)
|_{(v_{2\ell-1},v_{2\ell})=(x_{k+1},x_{k+1}^{-1})} \\
&\quad
=\Ri_{\bmu}(x_1,\dots,x_{k+1};v_1,\dots,v_{2\ell-2};a/t,b;t;p,q).
\end{align*}
Also, from \eqref{Eq_skew-factors} it follows that for $t^kab=pq$ the
following generalisation of \eqref{Eq_Cauchy-type} holds:
\begin{align}\label{Eq_Cauchy-type-2}
&\Ri_{\bmu}(x_1,\dots,x_k;v_1,\dots,v_{2\ell};a,b;t;p,q) \\
&\quad=
\Delta^0_{\bmu}\big(t^{k-1}a/b\big\vert t^{k-1}ax_1^{\pm},\dots,t^{k-1}ax_k^{\pm},
t^{k-1}a/v_1,\dots,t^{k-1}a/v_{2\ell};t,p,q\big). \notag
\end{align}
Finally, by \eqref{Eq_R-branching},
\begin{align}\label{Eq_R-branching-2}
&\Ri_{\bla}(x_1,\dots,x_k;v_1,v_2;at,b;t;p,q)\\
&\quad=\sum_{\bmu} 
\obinomE{\bla}{\bmu}_{[t^ka/b,tv_1v_2]
(t^ka/v_1,t^ka/v_2,pqa/tbv_1v_2);t;p,q} \,
\Ri_{\bmu}(x_1,\dots,x_k;a/v_1v_2,b;t;p,q). \notag
\end{align}

Recall our convention that parameters are assumed to be in
generic position.
Then both $\Ri_{\bmu}(x_1,\dots,x_k;a,b;t;p,q)$ and
$\Ri_{\bmu}(x_1,\dots,x_k;v_1,v_2;a,b;t;p,q)$ 
have sequences of poles in the complex $x_i$-plane converging to zero at
\begin{subequations}
\begin{equation}\label{Eq_seq1}
b^{-1}t^{1-j}q^{\mathbb{N}_0+1} p^{\ell},\quad
bt^{j-1}q^{\mathbb{N}_0} p^{-\ell},
\end{equation}
for $1\leq j\leq l(\mur{1})$, $1\leq \ell\leq \mur{1}_i$, and at
\begin{equation}\label{Eq_seq2}
b^{-1}t^{1-j}p^{\mathbb{N}_0+1} q^{\ell},\quad
bt^{j-1}p^{\mathbb{N}_0} q^{-\ell},
\end{equation}
\end{subequations}
for $1\leq j\leq l(\mur{2})$, $1\leq \ell\leq\mur{2}_i$.
By symmetry, it has diverging sequences of poles in the complex $x_i$-plane
at the reciprocals of the above points.

\subsection{The elliptic interpolation kernel}
We now turn our attention to the elliptic interpolation kernel, which was 
introduced by the second author in \cite{Rains18}.
The interpolation kernel generalises the elliptic interpolation
functions and has many remarkable properties, making it a powerful tool
for proving results for elliptic hypergeometric functions.
For more details the interested reader should consult \cite{LRW20,Rains18},
and for applications of the elliptic interpolation kernel to
elliptic hypergeometric integrals and dualities, see e.g., 
\cite{BHPS22a,BHPS22b,CHMPS22,HPS20,PRSZ20,Rains18}.

All the integrals described in this section are of the form
$\int f(z) \frac{\dup z}{z}$,
where $z:=(z_1,\dots,z_k)$, 
$\frac{\dup z}{z}:=\frac{\dup z_1}{z_1}\cdots \frac{\dup z_k}{z_k}$ and
$f(z)$ is $\mathrm{BC}_k$-symmetric.
Moreover, the contour of integration is assumed to always have the
product structure $C^k=C\times C\times\cdots\times C$, where
$C=C^{-1}$ is a positively oriented smooth Jordan curve around $0$
such that a given set of points $I_C$ lies in the interior of $C$.
For each of the integrals below we will explicitly describe this set.

For $x,y\in(\mathbb{C}^{\ast})^k$ and $c,t\in\mathbb{C}^{\ast}$, the elliptic 
interpolation kernel $\mathcal{K}_c(x;y;t;p,q)$ may be defined recursively 
by fixing one of the initial conditions
\[
\mathcal{K}_c(-;-;t;p,q)=1 \quad\text{or}\quad
\mathcal{K}_c(x_1;y_1;t;p,q)=
\frac{\Gammapq(cx_1^{\pm}y_1^{\pm})}{\Gammapq(t,c^2)},
\]
and imposing the branching rule
\begin{align}\label{Eq_branching-rule}
&\mathcal{K}_c(x_1,\dots,x_{k+1};y_1,\dots,y_{k+1};t;p,q)=
\frac{\prod_{i=1}^{k+1}\Gammapq(cx_i^{\pm}y_{k+1}^{\pm})}
{\Gammapq^{k+1}(t)\Gammapq(c^2)
\prod_{1\leq i<j\leq k+1}\Gammapq(tx_i^{\pm}x_j^{\pm})} \\
&\quad\times\int\mathcal{K}_{ct^{-1/2}}(z;y_1,\dots,y_k;t;p,q)
\Delta_{\mathrm{D}}\big(z;t^{1/2}x_1^{\pm},\dots,t^{1/2}x_{k+1}^{\pm},
pqy_{k+1}^{\pm}/ct^{1/2};p,q\big) \,\frac{\dup z}{z}, \notag
\end{align}
where $z:=(z_1,\dots,z_k)$ and $I_C$ is the union of the sets
\begin{gather*}
t^{-1}p^{\mathbb{N}_0+1} q^{\mathbb{N}_0+1}C,\quad
t^{1/2}x_i^{\pm} p^{\mathbb{N}_0} q^{\mathbb{N}_0} \; (1\leq i\leq k+1), \\
ct^{-1/2} y_i^{\pm} p^{\mathbb{N}_0} q^{\mathbb{N}_0} \; (1\leq i\leq k),\quad
c^{-1}t^{-1/2} y_{k+1}^{\pm} p^{\mathbb{N}_0+1} q^{\mathbb{N}_0+1}.
\end{gather*}
The condition that $t^{-1}p^{\mathbb{N}_0+1} q^{\mathbb{N}_0+1}C$ lies in 
the interior of $C$ (which can be dropped if $k=1$) requires that
$\abs{pq/t}<1$.
However, by the symmetry \cite[Proposition 3.5]{Rains09}
\[
\mathcal{K}_c(x;y;pq/t;p,q)
=\Gammapq^{2k}(t) \mathcal{K}_c(x;y;t;p,q)
\prod_{1\leq i<j\leq k}
\Gammapq\big(tx_i^{\pm}x_j^{\pm},ty_i^{\pm}y_j^{\pm}\big),
\]
the interpolation kernel may be meromorphically extended to
$t\in\mathbb{C}^{\ast}$.
Additional symmetries of the interpolation kernel, beyond the
$\mathrm{BC}_k$-symmetry in both $x$ and $y$, are
\[
\mathcal{K}_c(x;y;t;p,q)=\mathcal{K}_c(y;x;t;p,q)=\mathcal{K}_c(x;y;t;q,p)=
\mathcal{K}_{-c}(-x;y;t;p,q).
\]
Replacing $(c,x)\mapsto (-c,-x)$ in \eqref{Eq_branching-rule}
and using $\mathcal{K}_c(x;y;t;p,q)=\mathcal{K}_{-c}(-x;y;t;p,q)$, 
it follows that the branching rule, and hence the interpolation kernel, 
is independent of the choice of branch of $t^{1/2}$.
It should also be remarked that the symmetry in $y$ is not at all evident
from the definition and is a consequence of the same symmetry for
the formal interpolation kernel of \cite[Section~2]{Rains18}.

By specialising one of $x,y$ to $a\spec{\bla}_{k;t;p,q}/c$ 
for $\bla\in\mathscr{P}_k^2$ the interpolation kernel reduces to 
an elliptic interpolation function:
\begin{equation}\label{Eq_kernel-spec}
\mathcal{K}_c(x;a\spec{\bla}_{k;t;p,q}/c;t;p,q)
=\Ri_{\bla}(x;a,b;t;p,q)\prod_{i=1}^k
\frac{(pq/ab)^{2\lar{1}_i\lar{2}_i}\Gammapq(ax_i^{\pm},bx_i^{\pm})}
{\Gammapq(t^i,t^{i-1}ab)},
\end{equation}
where $b$ on the right is fixed by $c^2=t^{k-1}ab$.\footnote{Note that this
expression is independent of the choice of branch for $c$.}
The kernel also factors if $c=(pq/t)^{1/2}$ \cite[Proposition~2.10]{Rains18}:
\begin{equation}\label{Eq_kernel-c-factor}
\mathcal{K}_{(pq/t)^{1/2}}(x;y;t;p,q)
=\prod_{i,j=1}^k\Gammapq\big((pq/t)^{1/2}x_i^{\pm}y_j^{\pm}\big)
=\DeltaSe\big(x;y;(pq/t)^{1/2};p,q\big),
\end{equation}
where we recall the definition of 
$\DeltaSe\big(x;y;c;p,q\big)$ given in \eqref{Eq_DeltaSe}
(which does not necessarily assume that the alphabets $x$ and $y$
have the same cardinality).

The key property of the kernel from which our $\A_n$ integrals follow is 
\cite[Theorem~2.16]{Rains18}.

\begin{theorem}\label{Thm_kernel-decomp}
Let $k,\ell$ be nonnegative integers such that $k\leq \ell$, and
$b,c,d,t\in\mathbb{C}^{\ast}$, $x:=(x_1,\dots,x_{\ell})\in(\mathbb{C}^{\ast})^{\ell}$, 
$y:=(y_1,\dots,y_k)\in(\mathbb{C}^{\ast})^k$ such that $\abs{t},\abs{pq/t}<1$. 
Then
\begin{align}\label{Eq_kernel-decomp}
&\int\mathcal{K}_c(x;z,b,bt,\dots,bt^{\ell-k-1};t;p,q)
\mathcal{K}_d(z;y;t;p,q) 
\DeltaSv(z;t^{\ell-k}b,pq/bc^2d^2;t;p,q) \, \frac{\dup z}{z} \notag \\
&\quad =\mathcal{K}_{cd}(x;y_1,\dots,y_k,
bd,bdt,\dots,bdt^{\ell-k-1};t;p,q) \notag \\[1mm]
&\qquad \times \prod_{i=1}^{\ell-k}
\frac{\Gammapq(t^{1-i}c^2d^2)}{\Gammapq(t^{1-i}c^2)}
\prod_{i=1}^{\ell}\frac{\Gammapq(bcx_i^{\pm})}{\Gammapq(bcd^2x_i^{\pm})}
\prod_{i=1}^k\frac{\Gammapq(t^{\ell-k}bdy_i^{\pm})}
{\Gammapq(bc^2dy_i^{\pm})}, \notag
\end{align}
where $z:=(z_1,\dots,z_k)$ and $I_C$ is the union of the sets
\begin{gather*}
tp^{\mathbb{N}_0} q^{\mathbb{N}_0}C,\quad
t^{-1}p^{\mathbb{N}_0+1} q^{\mathbb{N}_0+1}C,\quad
t^{\ell-k}bp^{\mathbb{N}_0} q^{\mathbb{N}_0},\quad
(bc^2d^2)^{-1}p^{\mathbb{N}_0+1} q^{\mathbb{N}_0+1}, \\
cx_i^{\pm} p^{\mathbb{N}_0} q^{\mathbb{N}_0} \; (1\leq i\leq \ell),\quad
dy_i^{\pm} p^{\mathbb{N}_0} q^{\mathbb{N}_0} \; (1\leq i\leq k).
\end{gather*}
\end{theorem}

Specialising $c=(pq/t)^{1/2}$ we can use \eqref{Eq_kernel-c-factor} and
\eqref{Eq_reflection} to obtain the following corollary.

\begin{corollary}\label{Cor_kernel-decomp-2}
Let $k,\ell$ be nonnegative integers such that $k\leq\ell$, and
$b,d,t\in\mathbb{C}^{\ast}$, $x:=(x_1,\dots,x_{\ell})\in(\mathbb{C}^{\ast})^{\ell}$,
$y:=(y_1,\dots,y_k)\in(\mathbb{C}^{\ast})^k$ such that $\abs{t},\abs{pq/t}<1$.
Fix $c:=(pq/t)^{1/2}$.
Then
\begin{align*}
&\int \mathcal{K}_d(z;y;t;p,q) 
\DeltaSv(z;t^{\ell-k}b,t/bd^2;t;p,q) \DeltaSe\big(z;x;c;t;p,q\big) 
\,\frac{\dup z}{z} \\[1mm]
&\quad =\mathcal{K}_{cd}
\big(x;y_1,\dots,y_k,bd,bdt,\dots,bdt^{\ell-k-1};t;p,q\big) \\[1mm]
&\qquad \times 
\prod_{i=1}^{\ell-k} \frac{\Gammapq(t^i)}{\Gammapq(t^i/d^2)}\cdot
\frac{\prod_{i=1}^k\Gammapq\big(t^{\ell-k}bdy_i^{\pm},ty_i^{\pm}/bd\big)}
{\prod_{i=1}^{\ell}\Gammapq(bcd^2x_i^{\pm},ct^{k-\ell+1}x_i^{\pm}/b)},
\end{align*}
where $z:=(z_1,\dots,z_k)$ and 
$I_C$ is as in Theorem~\ref{Thm_kernel-decomp} with $c$ specialised accordingly.
\end{corollary}

If we further fix $\ell=k$, specialise $y=a\spec{\bmu}_{k;t;p,q}/d$
and make the substitution
\[
(a,b,d^2)\mapsto (t_1,t_3,t^{k-1}t_1t_2),
\]
we obtain the elliptic beta integral
\begin{align}\label{Eq_vdBult}
&\Int \Ri_{\bmu}(z;t_1,t_2;t;p,q) \DeltaSv\big(z;t_1,t_2,t_3,t_4,
cx_1^{\pm},\dots,cx_k^{\pm};t,p,q\big)\,\frac{\dup z}{z} \\[1mm]
&\quad=\Ri_{\bmu}(x;ct_1,ct_2;t;p,q) 
\Delta^0_{\bmu}\big(t^{k-1}t_1/t_2\vert t^{k-1}t_1t_3,t^{k-1}t_1t_4\big) 
\notag \\
&\qquad\times
\prod_{i=1}^k \bigg(\prod_{1\leq r<s\leq 4}
\Gammapq(t^{i-1}t_rt_s) \prod_{r=1}^4 \Gammapq(ct_rx_i^{\pm})\bigg), \notag
\end{align}
where $z:=(z_1,\dots,z_k)$, 
$t,t_1,t_2,t_3,t_4,x_1,\dots,x_k\in\mathbb{C}^{\ast}$ such that
$\abs{t}<1$ and $t^{k-2}t_1t_2t_3t_4=1$.
As before, $c:=(pq/t)^{1/2}$, and $I_C$ is the union of the sets
\[
t_r p^{\mathbb{N}_0} q^{\mathbb{N}_0} \; (1\leq r\leq 4),\quad
tp^{\mathbb{N}_0} q^{\mathbb{N}_0}C,\;\quad
cx_r^{\pm}p^{\mathbb{N}_0} q^{\mathbb{N}_0} \; (1\leq r\leq k),
\]
and the sets \eqref{Eq_seq1} and \eqref{Eq_seq2} with $b\mapsto t_2$.
For $\bmu=\bzero$ this is \cite[Theorem 3.1, $m=0$]{vdBult09} due to
van der Bult, see also \cite{Rains12,SV11}.

A final result for the interpolation kernel that is needed is
\cite[Corollary 3.25]{Rains18}.

\begin{proposition}\label{Prop_RK}
Let $\bmu\in\Part_k^2$, $x=(x_1,\dots,x_k)\in(\mathbb{C}^{\ast})^k$ and
$c,t,t_1,t_2,t_3,t_4,t_5\in \mathbb{C}^{\ast}$ such that
$\abs{t},\abs{pq/t}<1$ and 
\[
c^2t^{k-1}t_2t_3t_4t_5=pq.
\]
Then
\begin{align}\label{Eq_RK}
&\int \mathcal{K}_c(x;z;t;p,q)\Ri_{\bmu}(z;t_1,t_2;t;p,q)
\DeltaSv(z;t_2,t_3,t_4,t_5;t,p,q)\,\frac{\dup z}{z} \\
&\qquad=\prod_{i=1}^k\bigg(\prod_{2\leq r<s\leq 5}\Gamma(t^{i-1}t_rt_s)
\prod_{r=2}^5\Gamma\big(ct_rx_i^{\pm}\big)\bigg) \notag \\
&\qquad\quad\times\sum_{\bnu}\obinomE{\bmu}{\bnu}_{[t^{k-1}t_1/t_2,c^2]
(t^{k-1}t_1t_3,t^{k-1}t_1t_4,t^{k-1}t_1t_5);t;p,q}\Ri_{\bnu}(x;t_1/c,ct_2;t,p,q),
\notag
\end{align}
where $z:=(z_1,\dots,z_k)$ and $I_C$ is the union of
\begin{gather*}
tp^{\mathbb{N}_0} q^{\mathbb{N}_0}C,\quad
t^{-1}p^{\mathbb{N}_0+1} q^{\mathbb{N}_0+1}C,\quad
t_r p^{\mathbb{N}_0} q^{\mathbb{N}_0} \; (2\leq r\leq 5),\quad
cx_i^{\pm} p^{\mathbb{N}_0} q^{\mathbb{N}_0} \; (1\leq i\leq k)
\end{gather*}
and the sets \eqref{Eq_seq1} and \eqref{Eq_seq2} with $b\mapsto t_2$.
\end{proposition}

We will use this to prove the following key result.

\begin{theorem}\label{Thm_key}
Let $\bmu\in\Part^2$, $x=(x_1,\dots,x_k)\in(\mathbb{C}^{\ast})^k$ and
$c,t,t_1,t_2,t_3,t_4,v_1,v_2\in\mathbb{C}^{\ast}$ such that
\[
t_4=tv_1\quad\text{and}\quad c^2t^{k-1}t_1t_2t_3t_4=pq.
\]
Then
\begin{align*}
&\int \mathcal{K}_c(x;z;t;p,q)\Ri_{\bmu}(z;v_1,v_2;tt_1v_1v_2,t_2;t;p,q)
\DeltaSv\big(z;t_1,t_2,t_3,t_4;t;p,q\big)\,\frac{\dup z}{z} \\
&\quad=\prod_{i=1}^k\bigg(
\prod_{1\leq r<s\leq 4}\Gamma(t^{i-1}t_rt_s)
\prod_{r=1}^4\Gamma\big(ct_rx_i^{\pm}\big) \bigg) \\[1mm]
&\qquad \times
\frac{\Delta^0_{\bmu}(t^kt_1v_1v_2/t_2\vert t^kt_1v_1;t;p,q)}
{\Delta^0_{\bmu}(t^kt_1v_1v_2/t_2\vert c^2t^kt_1v_1;t;p,q)}\,
\Ri_{\bmu}(x;cv_1,v_2/c;ctt_1v_1v_2,ct_2;t;p,q),
\end{align*}
where $z:=(z_1,\dots,z_k)$ and 
$I_C$ is as in Proposition~\ref{Prop_RK} with $t_5\mapsto t_1$.
\end{theorem}

\begin{proof}[Proof of Theorem~\ref{Thm_key}]
In the following we write $\Gamma(x)$ instead of
$\Gammapq(x)$, $\mathcal{K}_c(x;z)$ instead of $\mathcal{K}_c(x;z;t;p,q)$,
and so on.

Setting $t_5=t_1$ in \eqref{Eq_RK}, which implies the balancing condition
$c^2t^{k-1}t_1t_2t_3t_4=pq$, gives
\begin{align*}
&\int \mathcal{K}_c(x;z)\Ri_{\bmu}(z;t_1,t_2)
\DeltaSv(z;t_1,t_2,t_3,t_4)\,\frac{\dup z}{z} \\
&\quad=\prod_{i=1}^k\bigg(
\prod_{1\leq r<s\leq 4}\Gamma(t^{i-1}t_rt_s)
\prod_{r=1}^4\Gamma\big(ct_rx_i^{\pm}\big)\bigg) \\
&\qquad\times 
\Delta^0_{\bmu}(t^{k-1}t_1/t_2\vert t^{k-1}t_1t_3,t^{k-1}t_1t_4)
\sum_{\bnu}\obinomE{\bmu}{\bnu}_{[t^{k-1}t_1/t_2,c^2](t^{k-1}t_1^2)}
\Ri_{\bnu}(x;t_1/c,ct_2) \\
&\quad=\prod_{i=1}^k\bigg(
\prod_{1\leq r<s\leq 4}\Gamma(t^{i-1}t_rt_s)
\prod_{r=1}^4\Gamma\big(ct_rx_i^{\pm}\big)\bigg) \\[2mm]
&\qquad \times
\Delta^0_{\bmu}(t^{k-1}t_1/t_2\vert t^{k-1}t_1t_3,t^{k-1}t_1t_4)
\Ri_{\bmu}(x;ct_1,ct_2),
\end{align*}
where the second equality follows from the connection coefficient
formula \eqref{Eq_connection}.
Multiplying both sides by
\[
\obinomE{\bla}{\bmu}_{[t^kt_1v_1v_2/t_2,tv_1v_2]
(t^kt_1v_1,t^kt_1v_2,pqt_1/tt_2)},
\]
where $\bla\in\Part^2$, and then summing over $\bmu$, yields
\begin{align*}
&\int \mathcal{K}_c(x;z)
\Ri_{\bla}(z;v_1,v_2;tt_1v_1v_2,t_2)
\DeltaSv(z;t_1,t_2,t_3,t_4)\,\frac{\dup z}{z} \\
&\quad=
\prod_{i=1}^k\bigg(\prod_{1\leq r<s\leq 4}\Gamma(t^{i-1}t_rt_s)
\prod_{r=1}^4\Gamma\big(ct_rx_i^{\pm}\big)\bigg) 
\Delta^0_{\bla}(t^kt_1v_1v_2/t_2\vert t^kt_1t_3v_1v_2,t^kt_1t_4v_1v_2) \\[1mm]
&\qquad \times \sum_{\bmu}
\obinomE{\bla}{\bmu}_{[t^kt_1v_1v_2/t_2,tv_1v_2]
(t^kt_1v_1,t^kt_1v_2,pqt_1/tt_2,pq/t_2t_3,pq/t_2t_4)}
\Ri_{\bmu}(x;ct_1,ct_2).
\end{align*}
Here the sum over $\bmu$ on the left has been carried out by
\eqref{Eq_R-branching-2} with $(a,b)\mapsto(t_1v_1v_2,t_2)$.
We now also assume that $t_4=tv_1$.
Then the sum on the right may be simplified 
by \eqref{Eq_Ebinom-reduction} with $w\mapsto t^kt_1v_1$ to
\begin{align*}
&\frac{\Delta^0_{\bla}(t^kt_1v_1v_2/t_2\vert t^kt_1v_1)}
{\Delta^0_{\bla}(t^kt_1v_1v_2/t_2\vert t^{k+1}t_1v_1^2v_2)} 
\sum_{\bmu} \obinomE{\bla}{\bmu}_{[t^kt_1v_1v_2/t_2,tv_1v_2]
(t^kc^2t_1v_1,t^kt_1v_2,pqt_1/tt_2)}\Ri_{\bmu}(x;ct_1,ct_2) \\[2mm]
&\quad=\frac{\Delta^0_{\bla}(t^kt_1v_1v_2/t_2\vert t^kt_1v_1)}
{\Delta^0_{\bla}(t^kt_1v_1v_2/t_2\vert t^{k+1}t_1v_1^2v_2)} \,
\Ri_{\bla}(x;tcv_1,v_2/c;ctt_1v_1v_2,ct_2),
\end{align*}
where the second equality follows from another application of
\eqref{Eq_R-branching-2}, now with
\[
(a,b,v_1,v_2)\mapsto(ct_1v_1v_2,ct_2,cv_1,v_2/c).
\]
As a result,
\begin{align*}
&\int \mathcal{K}_c(x;z) \Ri_{\bla}(z;v_1,v_2;tt_1v_1v_2,t_2)
\DeltaSv(z;t_1,t_2,t_3,t_4)\,\frac{\dup z}{z} \\
&\quad=\prod_{i=1}^k\bigg(\Gamma\big(ctv_1x_i^{\pm}\big)
\prod_{1\leq r<s\leq 3}\Gamma(t^{i-1}t_rt_s)
\prod_{r=1}^3 \Gamma\big(t^it_rv_1,ct_rx_i^{\pm}\big)\bigg) \\[1mm]
&\qquad\times
\Delta^0_{\bla}(t^kt_1v_1v_2/t_2\vert t^kt_1v_1,t^kt_1t_3v_1v_2)
\Ri_{\bla}(x;cv_1,v_2/c;ctt_1v_1v_2,ct_2).
\end{align*}
Replacing $\bla$ by $\bmu$ and applying the reflection equation 
\eqref{Eq_Delta-symm} completes the proof.
\end{proof}

\section{Proof and generalisations of Theorem~\ref{Thm_E-AnSelberg}}
\label{Sec_AFLT-proof}
The goal of this section is to prove the $\A_n$ elliptic Selberg integral of
Theorem~\ref{Thm_E-AnSelberg}.
As mentioned in the introduction, we will in fact prove an AFLT-type
generalisation of the theorem in which the integrand is multiplied by an
appropriate product of $\mathrm{BC}_n$-symmetric functions.

Throughout this section we suppress dependence on $p,q,t$.

\subsection{An $\A_n$ elliptic AFLT integral}
Before stating our main theorem we discuss the original AFLT integral of Alba,
Fateev, Litvinov and Tarnopolsky \cite{AFLT11} and some of its special cases
due to Kadell \cite{Kadell97} and Hua and Kadell~\cite{Hua63,Kadell93}.
For convenience these results will be expressed in terms of Selberg-type 
averages, and for $f\in\mathbb{C}[x_1,\dots,x_k]^{\Symm_k}=:\Lambda_k$, we
define 
\[
\big\langle f\big\rangle_{\alpha,\beta;\gamma}^k
:=\frac{1}{S_k(\alpha,\beta;\gamma)}
\Int_{[0,1]^k}\! f(x_1,\dots,x_k)
\prod_{i=1}^k x_i^{\alpha-1}(1-x_i)^{\beta-1}
\prod_{1\leq i<j\leq k}\abs{x_i-x_j}^{2\gamma}\,\dup x_1\cdots\dup x_k,
\]
where $S_k(\alpha,\beta;\gamma)$ is the Selberg integral \eqref{Eq_Selberg}.

For $\gamma\in\mathbb{C}^{\ast}$, let $P_{\la}^{(1/\gamma)}(x_1,\dots,x_k)$ 
be the Jack polynomial indexed by the partition $\la$,
see~\cite{Macdonald95,Stanley89}.
Also define the normalised Jack polynomial
\[
\tilde{P}_{\la}^{(1/\gamma)}(x_1,\dots,x_k):=
\frac{P_{\la}^{(1/\gamma)}(x_1,\dots,x_k)}{P_{\la}^{(1/\gamma)}(1,\dots,1)}.
\]
Then Kadell's generalised Selberg integral is \cite{Kadell97}
\begin{equation}\label{Eq_Kadell}
\big\langle\tilde{P}_{\la}^{(1/\gamma)}\big\rangle_{\alpha,\beta;\gamma}^k
=\prod_{i\geq 1}\frac{(\alpha+(k-i)\gamma)_{\la_i}}
{(\alpha+\beta+(2k-i-1)\gamma)_{\la_i}},
\end{equation}
where $(a)_n:=a(a+1)\cdots(a+n-1)$ is the ordinary shifted factorial.
In the case $\beta=\gamma$, Kadell further generalised this to a product of
two Jack polynomials as \cite{Kadell93}
\begin{equation}\label{Eq_Hua-Kadell}
\big\langle\tilde{P}_{\la}^{(1/\gamma)}\, \tilde{P}_{\mu}^{(1/\gamma)}
\big\rangle_{\alpha,\gamma;\gamma}^k
=\prod_{i,j=1}^k\frac{(\alpha+(2k-i-j)\gamma)_{\la_i+\mu_j}}
{(\alpha+(2k-i-j+1)\gamma)_{\la_i+\mu_j}}.
\end{equation}
Since in the Schur case, $\gamma=1$, this integral was previously discovered 
by Hua \cite{Hua63}, this last result is commonly referred to as the 
Hua--Kadell integral.

To describe the AFLT integral, which unifies \eqref{Eq_Kadell} and
\eqref{Eq_Hua-Kadell}, we need some basic plethystic notation,
see e.g.,~\cite{ARW21,Haglund08,Lascoux03}.
Let $\Lambda$ be the ring of symmetric functions in infinitely (but countably)
many variables over $\mathbb{C}$.
Then the power sum symmetric functions are defined as $p_0:=1$ and 
\[
p_r=x_1^r+x_2^r+\cdots,
\]
for $r\geq 1$.
Since $\Lambda=\mathbb{C}[p_1,p_2,\dots]$, any $f\in\Lambda$ admits an
expansion of the form $f=\sum_{\la} c_{\la} p_{\la}$, where
$p_{\la}=p_{\la_1}p_{\la_2}\cdots$.
Then for any $\xi\in\mathbb{C}$ and any alphabet $x$ (infinite or finite),
the expression $f[x+\xi]$ is defined as
\begin{equation}\label{Eq_pleth}
f[x+\xi]:=\sum_{\la} c_{\la}\prod_{i=1}^{l(\la)}\big(p_{\la_i}(x)+\xi\big).
\end{equation}
Clearly, if $x=(x_1,\dots,x_k)$ then $f[x+\xi]\in\Lambda_k$.
Moreover, $f[x]=f(x)$ and \eqref{Eq_pleth} unambiguously defines 
\[
f[k]=f(\underbrace{1,\dots,1}_{k \text{ times}}).
\]
Indeed, setting $x=\text{--}$ (the empty alphabet) and $\xi=k$ for
$k\in\mathbb{N}_0$ gives the same result as setting $x=(1,\dots,1)$ ($k$ ones) 
and $\xi=0$.

For $x=(x_1,\dots,x_k)$, let $\tilde{P}_{\la}[x+\xi]=P_{\la}[x+\xi]/P_{\la}[k+\xi]$.
Then the AFLT integral \cite[Appendix A]{AFLT11} may be stated as
\begin{align}\label{Eq_AFLT}
&\big\langle \tilde{P}_{\la}^{(1/\gamma)}[x]
\tilde{P}_{\mu}^{(1/\gamma)}[x+\beta/\gamma-1]
\big\rangle_{\alpha,\beta;\gamma}^k \\
&\quad=
\prod_{i=1}^k\frac{(\alpha+(k-i)\gamma)_{\la_i}}
{(\alpha+\beta+(2k-m-i-1)\gamma)_{\la_i}} 
\prod_{i=1}^k\prod_{j=1}^m
\frac{(\alpha+\beta+(2k-i-j-1)\gamma)_{\la_i+\mu_j}}
{(\alpha+\beta+(2k-i-j)\gamma)_{\la_i+\mu_j}}, \notag
\end{align}
where $\la\in\mathscr{P}_k$, $\mu\in\mathscr{P}$ and $m$ is any integer such
that $m\geq l(\mu)$.
The Kadell and Hua--Kadell integrals correspond to $\mu=0$ and $\beta=\gamma$
respectively.
As shown by Alba et al.~\cite{AFLT11}, the AFLT integral is important in
conformal field theory, particularly in the verification of the AGT conjecture
for $\mathrm{SU}(2)$, see \cite{AGT10}.
For further work on Selberg-type integrals and the AGT conjecture the
reader is referred to 
\cite{CPT20,FL12,IOY13,MMS11,MMS12,MMSS12,MS14,YHHZ23,ZM11}.

In our previous paper \cite{ARW21} we gave generalisations of the
AFLT integral to the elliptic level and to (non-elliptic) $\A_n$.
Our next theorem unifies these results by providing an elliptic $\A_n$
AFLT integral.
In the following we assume all the conditions of Theorem~\ref{Thm_E-AnSelberg}
to hold, including the fixing of a branch of $(pq/t)^{1/2}$.
For brevity we also suppress the dependence on $p,q$ and $t$ in most
of our functions, such as $\DeltaS(\dots;t;p,q)$, 
$\Ri_{\bla}(\dots;t;p,q)$ and $\Gammapq(z)$.

For $f:(\mathbb{C}^{\ast})^{k_1}\times\cdots\times(\mathbb{C}^{\ast})^{k_n}
\longrightarrow \mathbb{C}$ a function which is
$\mathrm{BC}_{k_r}$-symmetric in the $r$th set of variables, we define
the elliptic $\A_n$ Selberg average as
\begin{align*}
\big\langle f\big\rangle_{t_1,\dots,t_{2n+4}}^{k_1,\dots,k_n}
&:=\frac{1}{S_{k_1,\dots,k_n}^{\A_n}(t_1,\dots,t_{2n+4})} \\
&\quad\times \Int_C f\big(\zar{1},\dots,\zar{n}\big)
\DeltaS\big(\zar{1},\dots,\zar{n};t_1,\dots,t_{2n+4};(pq/t)^{1/2}\big)\,
\frac{\dup\zar{1}}{\zar{1}}\cdots\frac{\dup\zar{n}}{\zar{n}},
\end{align*}
where $S_{k_1,\dots,k_n}^{\A_n}(t_1,\dots,t_{2n+4})$ denotes
the elliptic $\A_n$ Selberg integral \eqref{Eq_E-AnSelberg}.
In addition to the conditions \eqref{Eq_in1} and \eqref{Eq_in2},
the contour $C=C_1^{k_1}\times\cdots\times C_n^{k_n}$ 
(where as before $C_r$ is a positively oriented smooth Jordan curve
around $0$ such that $C_r=C_r^{-1}$) should be such that 
any sequence of poles of $f$ in $\zar{r}_i$ tending to zero lies in the
interior of $C_r$, excluding those which are cancelled by the univariate part
of $\DeltaS(\zar{r};t_1,\dots,t_{2n+4})$.

\begin{theorem}[Elliptic $\A_n$ AFLT integral]\label{Thm_E-AnAFLT}
Assume the conditions of Theorem~\ref{Thm_E-AnSelberg} and let
$\tau_n:=t_{2n+1}t_{2n+2}t_{2n+3}/t^2$.
Then
\begin{align}\label{Eq_E-AnAFLT}
&\Big\langle\Ri_{\bla}\big(\zar{1};c^{1-n}t_1,c^{1-n}t_2\big)
\Ri_{\bmu}\big(\zar{n};t_{2n+2}/t,t_{2n+3}/t;t\tau_n,t_{2n+4}\big)
\Big\rangle_{t_1,\dots,t_{2n+4}}^{k_1,\dots,k_n}
\\ &\quad=
\prod_{r=3}^{2n}\Delta^0_{\bla}(t^{k_1-1}t_1/t_2\big\vert t^{k_1}t_1/t_r)
\prod_{r=2n+1}^{2n+4}
\Delta^0_{\bla}(t^{k_1-1}t_1/t_2\big\vert t^{k_1-1}t_1t_r) \notag \\[1mm]
&\qquad\times
\prod_{r=2n+2}^{2n+3}
\Delta^0_{\bmu}(t^{k_n}\tau_n/t_{2n+4}\vert t^{k_n-1}t_{2n+1}t_r)
\prod_{r=2}^n
\frac{\Delta^0_{\bmu}(t^{k_n}\tau_n/t_{2n+4}\vert t^{k_n}t_{2r-1}\tau_n)} 
{\Delta^0_{\bmu}(t^{k_n}\tau_n/t_{2n+4}\vert
t^{k_n+k_r-k_{r-1}}t_{2r-1}\tau_n)} \notag \\[1mm] 
&\qquad\times
\frac{\Delta^0_{\bmu}(t^{k_n}\tau_n/t_{2n+4}\vert 
t^{k_n}t_1\tau_n\spec{\bla}_{k_1;t;p,q})} 
{\Delta^0_{\bmu}(t^{k_n}\tau_n/t_{2n+4}\vert 
t^{k_n+1}t_1\tau_n\spec{\bla}_{k_1;t;p,q})},\notag
\end{align}
where $\bla\in\Part_{k_1}^2$ and $\bmu\in\Part^2$.
\end{theorem}

For $n=1$ the theorem reduces to the $\A_1$ elliptic AFLT integral
\cite[Theorem~1.4]{ARW21}.
In that paper we applied the symmetry-breaking trick introduced in
\cite{Rains09} to obtain the following AFLT integral for Macdonald
polynomials \cite[Corollary~1.5]{ARW21}:
\begin{align*}
&\frac{1}{k!(2\pi\iup)^k}\Int_{\mathbb{T}^k}
P_{\la}(z;q,t)
P_{\mu}\bigg(\bigg[z+\frac{t/c-b}{1-t}\bigg];q,t\bigg) \\
&\qquad\qquad\quad\times
\prod_{i=1}^k\frac{(a/z_i,qz_i/a;q)_{\infty}}{(b/z_i,cz_i;q)_{\infty}}
\prod_{1\leq i<j\leq k}\frac{(z_i/z_j,z_j/z_i;q)_{\infty}}
{(tz_i/z_j,tz_j/z_i;q)_{\infty}}\,\frac{\dup z}{z} \\[1mm]
&\quad=b^{\abs{\la}}(t/c)^{\abs{\mu}}
P_{\la}\bigg(\bigg[\frac{1-t^k}{1-t}\bigg];q,t\bigg)
P_{\mu}\bigg(\bigg[\frac{1-bct^{k-1}}{1-t}\bigg];q,t\bigg) \\
&\qquad\times 
\prod_{i=1}^k\frac{(t,act^{k-l(\mu)-i}q^{\la_i},at^{1-i}/b,
qt^{i-1}b/a;q)_{\infty}}
{(q,t^i,bct^{i-1},at^{1-i}q^{\la_i}/b;q)_{\infty}}
\prod_{i=1}^k\prod_{j=1}^{l(\mu)}
\frac{(act^{k-i-j+1}q^{\la_i+\mu_j};q)_{\infty}}
{(act^{k-i-j}q^{\la_i+\mu_j};q)_{\infty}}.
\end{align*}
Here $\la\in\Part_k$, $\mu\in\Part$ and $a,b,c\in\mathbb{C}^*$ such that
$\abs{b},\abs{c}<1$.\footnote{In \cite[Corollary~1.5]{ARW21} this was 
inadvertently stated with $c=1$, which would require a small indentation of 
the contour $\mathbb{T}$ at $1$.}
Thus far, we have not been able to replicate this procedure for the full
$\A_n$ elliptic AFLT integral, nor for the $\A_n$ elliptic Selberg integral
of Theorem~\ref{Thm_E-AnSelberg}.
However, one can show that under the natural generalisation of the limiting 
procedure of our previous paper the evaluation of either the elliptic $\A_n$
Selberg integral or elliptic AFLT integral reduce to $q$-analogues of their
ordinary counterparts. 
To be more specific, assume that $0<p,q<1$ and scale
the parameters $t_1,\dots,t_{2n+4}$ by
\[
(t_{2r-1},t_{2r})\mapsto(t_{2r-1},p^{1/2}t_{2r}),
\]
for $1\leq r\leq n$ and
\[
\big(t_{2n+1},t_{2n+2},t_{2n+3},t_{2n+4}\big)\mapsto
\big(t_{2n+1},p^{-1/4}t_{2n+2},p^{1/4}t_{2n+3},p^{1/2}t_{2n+4}\big).
\]
Then the $p\to0$ limit of the right-hand side of \eqref{Eq_E-AnSelberg}
exists and may be expressed as a product of $q$-shifted factorials. 
Now let $\alpha_1,\dots,\alpha_n,\beta,\gamma$ be as in the $\A_n$ 
Selberg integral \eqref{Eq_AnSelberg}.
By setting $t=q^\gamma$, $t_{2n+2}t_{2n+3}=q^\beta$ and
$t_{2r-1}t_{2n+1}=q^{\alpha_r+\cdots+\alpha_n+(r-n)\gamma}$ for
$1\leq r\leq n$, so that by the balancing conditions
\eqref{Eq_An-balancing} we have
$t_{2r}t_{2n+4}
=q^{1-\beta-\alpha_r-\cdots-\alpha_n-(k_r-k_{r-1}+k_n+r-n-2)\gamma}$,
one obtains a $q$-analogue of the $\A_n$ Selberg integral evaluation, up
to factor induced by the $q$-reflection formula for the $q$-gamma function.
Taking the $q\to1$ limit of this expression then produces the $\A_n$
Selberg integral evaluation up to a scalar.
The same procedure works for \eqref{Eq_E-AnAFLT}, but one additionally
needs the limit of the elliptic interpolation functions
\cite[Equations~(6.7)]{ARW21}.

Setting $\bmu=0$ in Theorem~\ref{Thm_E-AnAFLT} leads to the following 
generalisation of the Kadell integral.

\begin{corollary}[Elliptic $\A_n$ Kadell integral]
With the same conditions as Theorem~\ref{Thm_E-AnSelberg} and for 
$\bla\in\mathscr{P}_{k_1}^2$, 
\begin{align*}
&\Big\langle\Ri_{\bla}\big(\zar{1};c^{1-n}t_1,c^{1-n}t_2\big)
\Big\rangle_{t_1,\dots,t_{2n+4}}^{k_1,\dots,k_n} \\
&\quad =\prod_{r=3}^{2n}
\Delta^0_{\bla}(t^{k_1-1}t_1/t_2\big\vert t^{k_1}t_1/t_r)
\prod_{r=2n+1}^{2n+4}
\Delta^0_{\bla}(t^{k_1-1}t_1/t_2\vert t^{k_1-1}t_1t_r).
\end{align*}
\end{corollary}
Similarly, imposing the constraint $t_{2n+2}t_{2n+3}=t$ and using 
\eqref{Eq_v1v2=1} results in a generalisation of the Hua--Kadell integral.

\begin{corollary}[Elliptic $\A_n$ Hua--Kadell integral]
Assume the same conditions as in Theorem~\ref{Thm_E-AnSelberg} with the 
additional constraint $t_{2n+2}t_{2n+3}=t$
Then, for $\bla,\bmu\in\mathscr{P}_{k_1}^2$,
\begin{align*}
&\Big\langle\Ri_{\bla}\big(\zar{1};c^{1-n}t_1,c^{1-n}t_2\big)
\Ri_{\bmu}\big(\zar{n};t_{2n+1},t_{2n+4}\big)
\Big\rangle_{t_1,\dots,t_{2n+4}}^{k_1,\dots,k_n} \\ 
&\quad=\prod_{r=3}^{2n}
\Delta^0_{\bla}(t^{k_1-1}t_1/t_2\vert t^{k_1}t_1/t_r) 
\prod_{r=2n+1}^{2n+4}
\Delta^0_{\bla}(t^{k_1-1}t_1/t_2\big\vert t^{k_1-1}t_1t_r) \\[1mm]
&\qquad\times \prod_{r=2n+2}^{2n+3}
\Delta^0_{\bmu}(t^{k_n-1}t_{n+1}/t_{2n+4}\vert t^{k_n-1}t_{2n+1}t_r) \\[1mm]
&\qquad\times\prod_{r=2}^n
\frac{\Delta^0_{\bmu}(t^{k_n-1}t_{2n+1}/t_{2n+4}\vert
t^{k_n-1}t_{2r-1}t_{2n+1})} 
{\Delta^0_{\bmu}(t^{k_n-1}t_{2n+1}/t_{2n+4}\vert
t^{k_n+k_r-k_{r-1}-1}t_{2r-1}t_{2n+1})} \\[1mm] 
&\qquad\times
\frac{\Delta^0_{\bmu}(t^{k_n-1}t_{2n+1}/t_{2n+4}\vert 
t^{k_n-1}t_1t_{2n+1}\spec{\bla}_{k_1;t;p,q})} 
{\Delta^0_{\bmu}(t^{k_n-1}t_{2n+1}/t_{2n+4}\vert 
t^{k_n}t_1t_{2n+1}\spec{\bla}_{k_1;t;p,q})}.
\end{align*}
\end{corollary}

\subsection{Proof of Theorems~\ref{Thm_E-AnSelberg} and \ref{Thm_E-AnAFLT}}

Let $0\leq k_1\leq k_2\leq\cdots\leq k_n$, $c:=(pq/t)^{1/2}$
(with a branch of $c$ fixed) and let  $t_1,\dots,t_{2n+4}$ satisfy the
balancing conditions \eqref{Eq_An-balancing}, i.e.,
\[
t^{k_1+k_n-2}t_1t_2t_{2n+1}t_{2n+2}t_{2n+3}t_{2n+4}=pq
\]
and
\[
t^{k_r-k_{r-1}+k_n-2}t_{2r-1}t_{2r}t_{2n+1}t_{2n+2}t_{2n+3}t_{2n+4}=pq
\]
for $2\leq r\leq n$.
The reason for restating these conditions as per the above, separating out
the $r=1$ case, is that in what follows we will introduce an integer $k_0$
which, unlike in Theorem~\ref{Thm_E-AnSelberg}, will not be $0$.

The task is to evaluate the integral
\begin{align}\label{Eq_Slamu}
&S^{k_1,\dots,k_n}_{\bla,\bmu}(t_1,\dots,t_{2n+4}) \\[1mm]
&\quad:=\Int \Big(\Ri_{\bla}\big(\zar{1};c^{1-n}t_1,c^{1-n}t_2\big) 
\Ri_{\bmu}\big(\zar{n};t_{2n+2}/t,t_{2n+3}/t;t\tau_n,t_{2n+4}\big) 
\notag \\
&\qquad\qquad \times
\DeltaS\big(\zar{1},\dots,\zar{n};t_1,\dots,t_{2n+4};c\big)\Big)\,
\frac{\dup\zar{1}}{\zar{1}}\cdots\frac{\dup\zar{n}}{\zar{n}},
\notag
\end{align}
where $\tau_n:=t_{2n+1}t_{2n+2}t_{2n+3}/t^2$.
To this end we consider the more general problem of evaluating
\begin{align*}
&S^{k_0,k_1,\dots,k_n}_{\bmu}(x;t_1,\dots,t_{2n+4}) \\[1mm] 
&\quad := \Int \Bigg( \mathcal{K}_d\big(\zar{1};x\big)
\Ri_{\bmu}\big(\zar{n};t_{2n+2}/t,t_{2n+3}/t;t\tau_n,t_{2n+4}\big) \\
&\qquad\qquad\; \times 
\frac{\DeltaS\big(\zar{1},\dots,\zar{n};t_1,\dots,t_{2n+4};c\big)}
{\prod_{i=1}^{k_1}\prod_{r=1}^2\Gamma\big(c^{1-n}t_r(\zar{1}_i)^{\pm}\big)}
\Bigg)\,
\frac{\dup\zar{1}}{\zar{1}}\cdots\frac{\dup\zar{n}}{\zar{n}}.
\end{align*}
Here $k_0,k_1,\dots,k_n$ are integers such that
$0\leq k_1\leq\cdots\leq k_n$, $x:=(x_1,\dots,x_{k_1})$, 
\begin{equation}\label{Eq_d}
d^2:=c^{2-2n}t^{k_1-k_0-1}t_1t_2
\end{equation}
and the $t_1,\dots,t_{2n+4}$ satisfy the modified balancing conditions
\begin{equation}\label{Eq_An-balancing-2}
t^{k_r-k_{r-1}+k_n-2}t_{2r-1}t_{2r}t_{2n+1}t_{2n+2}t_{2n+3}t_{2n+4}=pq
\end{equation}
for all $1\leq r\leq n$.
By \eqref{Eq_kernel-spec} and \eqref{Eq_reflection},
\begin{align}\label{Eq_xtonox}
S^{k_1,\dots,k_n}_{\bla,\bmu}(t_1,\dots,t_{2n+4}) 
&=\prod_{i=1}^{k_1}
(c^{2n-2}pq/t_1t_2)^{-2\lar{1}_i\lar{2}_i}
\Gamma(t^i,c^{2-2n}t^{i-1}t_1t_2) \\
&\quad\times
S^{0,k_1,\dots,k_n}_{\bmu}\big(c^{1-n}t_1\spec{\bla}_{k_1}/d;
t_1,\dots,t_{2n+4}\big), \notag
\end{align}
where $d$ on the right is given by \eqref{Eq_d} with $k_0=0$.

\begin{proposition}\label{Prop_xSelberg}
With the parameters satisfying the conditions
\eqref{Eq_d} and \eqref{Eq_An-balancing-2},
\begin{align}\label{Eq_Prop_xSelberg}
&S^{k_0,k_1,\dots,k_n}_{\bmu}(x;t_1,\dots,t_{2n+4}) \\
&\quad=\prod_{i=1}^{k_1} 
\bigg(\Delta^0_{\bmu}\big(t^{k_n}\tau_n/t_{2n+4}\big\vert 
t^{k_n}c^{n-1}d\tau_n x_i^{\pm}\big) \notag \\
&\qquad\qquad\times
\prod_{r=3}^{2n} \Gamma\big(c^{n-1}dtx_i^{\pm}/t_r\big)
\prod_{r=2n+1}^{2n+4} \Gamma\big(c^{n-1}dt_rx_i^{\pm}\big)\bigg) 
\notag \\
&\qquad\times
\prod_{r=2}^n \prod_{i=1}^{k_r-k_{r-1}}
\Gamma(t^i,t^{i-1}c^{2r-2n}t_{2r-1}t_{2r})
\prod_{2n+1\leq r<s\leq 2n+4}\,\prod_{i=1}^{k_n}\Gamma(t^{i-1}t_rt_s) 
\notag \\
&\qquad\times
\prod_{2\leq r<s\leq n}
\prod_{i=1}^{k_r-k_{r-1}}
\Gamma(t^it_{2r-1}/t_{2s-1},t^it_{2r}/t_{2s-1},
t^it_{2r-1}/t_{2s},t^it_{2r}/t_{2s}) \notag \\
&\qquad\times
\prod_{r=2}^n \prod_{s=2n+1}^{2n+4} \prod_{i=1}^{k_r-k_{r-1}}
\Gamma(t^{i-1}t_{2r-1}t_s,t^{i-1}t_{2r}t_s) \notag \\
&\qquad \times
\prod_{r={2n+2}}^{2n+3}\Delta^0_{\bmu}
(t^{k_n}\tau_n/t_{2n+4}\vert t^{k_n-1}t_{2n+1}t_r) 
\notag \\
&\qquad \times
\prod_{r=2}^n \frac{\Delta^0_{\bmu}
(t^{k_n}\tau_n/t_{2n+4}\vert t^{k_n}t_{2r-1}\tau_n)}
{\Delta^0_{\bmu}(t^{k_n}\tau_n/t_{2n+4}\vert 
t^{k_n+k_r-k_{r-1}}t_{2r-1}\tau_n)}. \notag
\end{align}
\end{proposition}

It is readily checked using \eqref{Eq_xtonox} that by this implies
Theorems~\ref{Thm_E-AnSelberg} and \ref{Thm_E-AnAFLT}.
In particular, from \eqref{Eq_d} and the $r=1$ case of
\eqref{Eq_An-balancing-2}, $pq=t^{k_n-1}c^{2n-2}d^2\tau_n t_{2n+4}$.
Combined with \eqref{Eq_Delta-symm} this yields
\begin{align*}
&\prod_{i=1}^{k_1} 
\Delta^0_{\bmu}\big(t^{k_n}\tau_n/t_{2n+4}\big\vert 
t^{k_n}c^{n-1}d\tau_n x_i^{\pm}\big)\big|_
{x_i\mapsto c^{1-n}t_1(\spec{\bla}_{k_1})_i/d} \\
&\quad=\frac{\Delta^0_{\bmu}(t^{k_n}\tau_n/t_{2n+4}\vert 
t^{k_n}t_1\tau_n\spec{\bla}_{k_1})} 
{\Delta^0_{\bmu}(t^{k_n}\tau_n/t_{2n+4}\vert 
t^{k_n+1}t_1\tau_n\spec{\bla}_{k_1})}.
\end{align*}
Furthermore, by the same specialisation of the $x_i$,
\eqref{Eq_d}, \eqref{Eq_An-balancing-2} and \eqref{Eq_Gamma-Delta} with
\[
n\mapsto k_1,\; a\mapsto t^{k_1}t_1/t_r,\;
b\mapsto c^{2n-2}d^2t^{2-k_1}/t_1t_r=t^{1-k_0}t_2/t_r,
\]
and
\[
n\mapsto k_1,\; a\mapsto t^{k_1-1}t_1t_r,\;
b\mapsto c^{2n-2}d^2t^{1-k_1}t_r/t_1=t^{-k_0}t_2t_r,
\]
respectively, we get
\begin{align*}
\prod_{i=1}^{k_1} 
\prod_{r=3}^{2n} \Gamma\big(c^{n-1}dtx_i^{\pm}/t_r\big)
&\mapsto
(c^{2n-2}t^{k_1-k_n})^{2\sum_{i=1}^{k_1}\lar{1}_i\lar{2}_i} \\
&\quad\times
\prod_{r=3}^{2n} \Gamma(t^it_1/t_r,t^{i-k_0}t_2/t_r)
\Delta^0_{\bla}(t^{k_0+k_1-1}t_1/t_2\vert t^{k_1}t_1/t_r) 
\end{align*}
and
\begin{align*}
\prod_{i=1}^{k_1} 
\prod_{r=2n+1}^{2n+4} \Gamma\big(c^{n-1}dt_rx_i^{\pm}\big) 
&\mapsto
\Big(\frac{pqt^{k_0-k_1+k_n}}{t_1t_2}\Big)^
{2\sum_{i=1}^{k_1}\lar{1}_i\lar{2}_i} \\
&\quad\times
\prod_{r=2n+1}^{2n+4} 
\Gamma\big(t^{i-1}t_1t_r,t^{i-k_0-1}t_2t_r\big)
\Delta^0_{\bla}(t^{k_0+k_1-1}t_1/t_2\vert t^{k_1-1}t_1t_r).
\end{align*}
Combining these three results, setting $k_0=0$ and using 
\eqref{Eq_xtonox} implies Theorems~\ref{Thm_E-AnSelberg}
and~\ref{Thm_E-AnAFLT}.

\begin{proof}[Proof of Proposition~\ref{Prop_xSelberg}]
Recalling the definition of the $\A_n$ Selberg density
\eqref{Eq_An-S-density}, and assuming that $n\geq 2$, we have
\begin{align*}
&S^{k_0,k_1,\dots,k_n}_{\bmu}(x;t_1,\dots,t_{2n+4}) \\
&\quad=
\Int \Big( \mathcal{K}_d\big(\zar{1};x\big)
\Ri_{\bmu}\big(\zar{n};t_{2n+2}/t,t_{2n+3}/t;t\tau_n,t_{2n+4}\big) \\
&\qquad\qquad \times 
\DeltaSv\big(\zar{1};c^{n-1}t/t_3,c^{n-1}t/t_4\big)
\DeltaSe\big(\zar{1};\zar{2};c\big) \\
&\qquad\qquad \times 
\DeltaS\big(\zar{2},\dots,\zar{n};t_3,\dots,t_{2n+4};c\big)
\Big)\,
\frac{\dup\zar{1}}{\zar{1}}\cdots\frac{\dup\zar{n}}{\zar{n}}.
\end{align*}
By Corollary~\ref{Cor_kernel-decomp-2} with $d$ as given in
\eqref{Eq_d} and
\[
(k,\ell,b,x,y,z)\mapsto 
\big(k_1,k_2,c^{n-1}t^{k_1-k_2+1}/t_3,\zar{2},x,\zar{1}\big),
\]
we can carry out the integration over $\zar{1}$.
In particular we note that the above substitutions imply that
\[
t/bd^2\mapsto c^{n-1}t^{k_0-2k_1+k_2+1}t_3/t_1t_2=c^{n-1}t/t_4,
\]
where the last equality follows from by taking the ratio of the
balancing conditions \eqref{Eq_An-balancing-2} for $r=1$ and $r=2$.
From these same balancing conditions it also follows that
\[
(cd)^2=c^{4-2n}t^{k_2-k_1-1}t_3t_4.
\]
As a result,
\begin{align*}
&S^{k_0,k_1,\dots,k_n}_{\bmu}(x;t_1,\dots,t_{2n+4}) \\
&\quad=\prod_{i=1}^{k_2-k_1}\Gamma(t^i,t^{i-1}c^{4-2n}t_3t_4)
\prod_{i=1}^{k_1}
\Gamma\big(c^{n-1}dtx_i^{\pm}/t_3,c^{n-1}dtx_i^{\pm}/t_4\big)\\[1mm]
&\qquad\times
S^{k_1,k_2,\dots,k_n}_{\bmu}(x';t_3,\dots,t_{2n+4}),
\end{align*}
where
\[
x':=\big(x_1,\dots,x_{k_1},c^{n-1}dt^{k_1-k_2+1}/t_3,
c^{n-1}dt^{k_1-k_2+2}/t_3,\dots,c^{n-1}d/t_3\big).
\]
A straightforward but somewhat tedious calculations shows that the
right-hand side of \eqref{Eq_Prop_xSelberg} satisfies the same recursion.
The proof is thus reduced to checking validity of the claim for $n=1$.
This is
\[
S^{k_0,k_1}_{\bmu}(x;t_1,\dots,t_6)=
\int \mathcal{K}_d\big(z;x\big)
\Ri_{\bmu}\big(z;t_4/t,t_5/t;t\tau_1,t_6\big) 
\DeltaSv\big(z;t_3,t_4,t_5,t_6\big)\,
\frac{\dup z}{z},
\]
where $z=(z_1,\dots,z_{k_1})$, $d^2:=t^{k_1-k_0-1}t_1t_2$ and
$t^{2k_1-k_0-2}t_1\cdots t_6=pq$.
But this is nothing but Theorem~\ref{Thm_key} with
\[
(n,c,t_1,t_2,t_3,v_1,v_2)\mapsto (k_1,d,t_3,t_6,t_5,t_4/t,t_5/t).
\]
Hence
\begin{align*}
S^{k_0,k_1}_{\bmu}(x;t_1,\dots,t_6)
&=\prod_{i=1}^{k_1}\bigg(
\prod_{3\leq r<s\leq 6} \Gamma(t^{i-1}t_rt_s)
\prod_{r=3}^6 \Gamma\big(dt_rx_i^{\pm}\big) \bigg) \\[1mm]
&\quad\times
\frac{\Delta^0_{\bmu}(t^{k_1-2}t_3t_4t_5/t_6\vert t^{k_1-1}t_3t_4)}
{\Delta^0_{\bmu}(t^{k_1-2}t_3t_4t_5/t_6\vert d^2t^{k_1-1}t_3t_4)}\,
\Ri_{\bmu}(x;dt_4/t,t_5/td;dt_3t_4t_5/t,dt_6).
\end{align*}
Since
\[
t^{k_1}(dt_3t_4t_5/t)(dt_6)=t^{2k_1-k_0-2}t_1t_2t_3t_4t_5t_6=pq,
\]
the interpolation function on the right is of Cauchy type and
factors by \eqref{Eq_Cauchy-type-2}.
Therefore,
\begin{align*}
S^{k_0,k_1}_{\bmu}(x;t_1,\dots,t_6)
&=\prod_{i=1}^{k_1}\bigg(
\prod_{3\leq r<s\leq 6} \Gamma(t^{i-1}t_rt_s)
\prod_{r=3}^6 \Gamma\big(dt_rx_i^{\pm}\big) \bigg) \\[1mm]
&\quad\times
\prod_{r=4}^5\Delta^0_{\bmu}(t^{k_1}\tau_1/t_6\vert t^{k_1-1}t_3t_r)
\prod_{i=1}^{k_1}
\Delta^0_{\bmu}\big(t^{k_1}\tau_1/t_6\vert t^{k_1}d\tau_1 x_i^{\pm}\big).
\end{align*}
This is exactly the right-hand side of \eqref{Eq_Prop_xSelberg}
for $n=1$.
\end{proof}

To conclude this section we remark that for $k_1=k_2=\dots=k_n=k$
the evaluation of \eqref{Eq_Slamu} that does not require the heavy
machinery of the elliptic interpolation kernel.
As per the above proof, for $n\geq 2$,
\begin{align*}
&S^{k_1,\dots,k_n}_{\bla,\bmu}(t_1,\dots,t_{2n+4}) \\
&\quad=
\Int \Big( 
\Ri_{\bla}\big(\zar{1};c^{1-n}t_1,c^{1-n}t_2\big) 
\Ri_{\bmu}\big(\zar{n};t_{2n+2}/t,t_{2n+3}/t;t\tau_n,t_{2n+4}\big) \\
&\qquad\qquad \times 
\DeltaSv\big(\zar{1};c^{1-n}t_1,t^{1-n}t_2,c^{n-1}t/t_3,c^{n-1}t/t_4\big)
\DeltaSe\big(\zar{1};\zar{2};c\big) \\
&\qquad\qquad \times 
\DeltaS\big(\zar{2},\dots,\zar{n};t_3,\dots,t_{2n+4};c\big)
\Big)\,
\frac{\dup\zar{1}}{\zar{1}}\cdots\frac{\dup\zar{n}}{\zar{n}}.
\end{align*}
If $k_1=k_2=k$ then the integral over $\zar{1}$ is exactly
the elliptic beta integral \eqref{Eq_vdBult} with
\[
(t_1,t_2,t_3,t_4)\mapsto(c^{1-n}t_1,c^{1-n}t_2,c^{n-1}t/t_3,c^{n-1}t/t_4),
\]
$\bmu\mapsto\bla$ and $x_i\mapsto\zar{2}_i$ for $1\leq i\leq k$.
In particular, by \eqref{Eq_An-balancing}, $t^k t_1t_2/t_3t_4=1$, as required.
Therefore,
\begin{align*}
S^{k,k,k_3,\dots,k_n}_{\bla,\bmu}(t_1,\dots,t_{2n+4})&=
\prod_{i=1}^k \bigg(\frac{\Gamma(t^{i-1}c^{2-2n}t_1t_2)}
{\Gamma(t^{i-1}c^{4-2n}t_1t_2)} 
\prod_{r=1}^2\prod_{s=3}^4 \Gamma(t^it_r/t_s)\bigg)\\
& \quad \times
\Delta^0_{\bla}(t^{k-1}t_1/t_2\big\vert t^k t_1/t_3,t^kt_1/t_4)
S^{k,k_3,\dots,k_n}_{\bla,\bmu}(t_1,t_2,t_5,\dots,t_{2n+4}),
\end{align*}
where $t_3t_4=t^kt_1t_2$.
For $k=1$ and $\bla=\bzero$ this is the elliptic analogue of the recursion
at the bottom of page~299 of~\cite{Warnaar09}.
Iterating the recursion yields
\begin{align*}
S^{\overbrace{\scriptstyle{k,\dots,k}}^{m\text{ times}}
\!\!,k_{m+1},\dots,k_n}_ {\bla,\bmu}(t_1,\dots,t_{2n+4})&=
\prod_{i=1}^k\bigg(\frac{\Gamma(t^{i-1}c^{2-2n}t_1t_2)}
{\Gamma(t^{i-1}c^{2m-2n}t_1t_2)} 
\prod_{r=1}^2 \prod_{s=3}^{2m} \Gamma(t^it_r/t_s)\bigg)\\
& \quad \times
\prod_{r=3}^{2m}\Delta^0_{\bla}(t^{k-1}t_1/t_2\vert t^k t_1/t_r) \\
& \quad \times
S^{k,k_{m+1},\dots,k_n}_{\bla,\bmu}(t_1,t_2,t_{2m+1},\dots,t_{2n+4};t;p,q),
\end{align*}
where $1\leq m\leq n$ and $t_{2m-1}t_{2m}=\cdots=t_3t_4=t^kt_1t_2$.
In particular, for $m=n$,
\begin{align*}
S^{k,\dots,k}_{\bla,\bmu}(t_1,\dots,t_{2n+4})
&=\prod_{i=1}^k \bigg(\frac{\Gamma(t^{i-1}c^{2-2n}t_1t_2)}{\Gamma(t^{i-1}t_1t_2)} 
\prod_{r=1}^2 \prod_{s=3}^{2n} \Gamma(t^it_r/t_s)\bigg)\\
& \quad \times
\prod_{r=3}^{2n}\Delta^0_{\bla}(t^{k-1}t_1/t_2\vert t^k t_1/t_r) \\
& \quad \times
S^k_{\bla,\bmu}(t_1,t_2,t_{2n+1},\dots,t_{2n+4};t;p,q)
\end{align*}
This final integral is the elliptic AFLT integral of \cite[Theorem 1.4]{ARW21},
evaluated in \cite{ARW21} without the use of the interpolation kernel.
Hence
\begin{align*}
S^{k,\dots,k}_{\bla,\bmu}(t_1,\dots,t_{2n+4})
&=\prod_{i=1}^k
\bigg( \Gamma(t^i,t^{i-1}c^{2-2n}t_1t_2)
\prod_{r=1}^2 \prod_{s=3}^{2n} \Gamma(t^it_r/t_s) \\
&\qquad\quad\times\prod_{r=1}^2 \prod_{s=2n+1}^{2n+4} \Gamma(t^{i-1}t_rt_s)
\prod_{2n+1\leq r<s\leq 2n+4} \Gamma(t^{i-1}t_rt_s)\bigg) \\
& \quad \times
\prod_{r=3}^{2n}\Delta^0_{\bla}(t^{k-1}t_1/t_2\vert t^k t_1/t_r) 
\prod_{r=2n+1}^{2n+4} \Delta^0_{\bla}(t^{k-1}t_1/t_2\vert t^{k-1}t_1t_r) \\[1mm]
&\quad\times
\prod_{r=2n+2}^{2n+3}
\Delta^0_{\bmu}(t^{k_n}\tau_n/t_{2n+4}\vert t^{k-1}t_{2n+1}t_r) \\
&\quad\times
\frac{\Delta^0_{\bmu}(t^k\tau_n/t_{2n+4}\vert 
t^kt_1\tau_n\spec{\bla}_k)} 
{\Delta^0_{\bmu}(t^k\tau_n/t_{2n+4}\vert t^{k+1}t_1\tau_n\spec{\bla}_{k;t})},
\end{align*}
where
\begin{align*}
t^{2k-2}t_1t_2t_{2n+1}t_{2n+2}t_{2n+3}t_{2n+4}
&=t^{k-2}t_3t_4t_{2n+1}t_{2n+2}t_{2n+3}t_{2n+4} \\
&=\cdots=t^{k-2}t_{2n-1}t_{2n}t_{2n+1}t_{2n+2}t_{2n+3}t_{2n+4}=pq.
\end{align*}

\end{document}